\documentclass[11pt]{article}
\usepackage{latexsym,amsfonts,amsmath,amsthm,amssymb, epsfig}

\usepackage{subfig}

\usepackage[vcentering,dvips]{geometry}
\newtheorem{thm}{Theorem}
\newtheorem{lem}[thm]{Lemma}
\newtheorem{prop}[thm]{Proposition}

\theoremstyle{remark}
\newtheorem{rem}{Remark}

\theoremstyle{definition}
\newtheorem{dfn}[thm]{Definition}

\newcommand\R{\mathbb{R}}
\newcommand\N{\mathbb{N}}

\renewcommand\S{\mathbb{S}}
\newcommand\D{\Delta_p^n}
\newcommand\supp{{\rm supp\ }}

\newcommand\E{\mathbb{E}}

\newlength{\fixboxwidth}
\title{Average best $m$-term approximation}

\author{
Jan Vyb\'\i ral\footnote{Johann Radon Institute for Computational and
Applied Mathematics, Austrian Academy of Sciences, Altenbergerstrasse 69, A-4040 Linz, Austria,
email: {\tt  jan.vybiral@oeaw.ac.at}, Tel: +43 732 2468 5262, Fax: +43 732 2468 5212.}
}

\begin{document}

\maketitle

\begin{abstract}
We introduce the concept of average best $m$-term approximation widths with respect to 
a probability measure on the unit ball or the unit sphere of $\ell_p^n$.
We estimate these quantities for the embedding $id:\ell_p^n\to\ell_q^n$ 
with $0<p\le q\le \infty$ for the normalized cone and surface measure.
Furthermore, we consider certain tensor product weights and show that
a typical vector with respect to such a measure exhibits a strong compressible
(i.e. nearly sparse) structure. This measure may be therefore used
as a random model for sparse signals.
\end{abstract}

\noindent{\bf AMS subject classification (MSC 2010):} Primary: 41A46, Secondary: 52A20, 60B11, 94A12.


\section{Introduction}
\subsection{Best $m$-term approximation}
Let $m\in \N_0$ and let $\Sigma_m$ be the set of all sequences $x=\{x_j\}_{j=1}^\infty$ with
$$
\|x\|_{0}:=\#\,\supp x=\#\{ n\in\N: x_n\not = 0\}\le m.
$$
Here stands $\#A$ for the number of elements of a set $A$. The elements of $\Sigma_m$
are said to be \emph{$m$-sparse}.
Observe, that $\Sigma_m$ is a non-linear subset of every 
$\ell_q:=\{x=\{x_j\}_{j=1}^\infty: \|x\|_q<\infty\}$, where
$$
\|x\|_{q}:=\begin{cases}
\left(\sum_{j=1}^\infty |x_j|^q\right)^{1/q},&\quad 0<q<\infty,\\
\sup_{j\in \N}|x_j|,&\quad q=\infty.
\end{cases}
$$
For every $x\in\ell_q$, we define its \emph{best $m$-term approximation error} by
$$
\sigma_{m}(x)_q:=\inf_{y\in \Sigma_m}\|x-y\|_q.
$$
Moreover for $0<p\le q\le \infty$, we introduce the \emph{best $m$-term approximation widths}
$$
\sigma^{p,q}_{m}:=\sup_{x:\|x\|_p\le 1}\sigma_m(x)_q.
$$
The use of this concept goes back to Schmidt \cite{S} and after the work of Oskolkov \cite{O},
it was widely used in the approximation theory, cf. \cite{DNS,DJP,T}. In fact, it is the main prototype of nonlinear approximation \cite{D}. 
It is well known, that 
\begin{equation}\label{eq:1}
2^{-1/p}(m+1)^{1/q-1/p}\le \sigma^{p,q}_m\le (m+1)^{1/q-1/p},\quad m=0,1,2,\dots.
\end{equation}
The proof of \eqref{eq:1} is based on the simple fact, that (roughly speaking)
the best $m$-term approximation error of $x\in\ell_p$ is realized by subtracting the $m$ largest coefficients taken in absolute value.
Hence,
$$
 \sigma_m(x)_q = \begin{cases} \biggl(\sum_{j=m+1}^\infty (x^*_{j})^q\biggr)^{1/q},&\quad 0<q<\infty,\\
 x_{m+1}^*=\sup_{j\ge m+1} x_j^*,&\quad q=\infty,
 \end{cases}
$$
where $x^*=(x_1^*,x_2^*,\dots)$ denotes the so-called {\it non-increasing rearrangement} \cite{BS} of the vector $(|x_1|, |x_2|, |x_3|,\dots)$.

Let us recall the proof of \eqref{eq:1} in the simplest case, namely $q=\infty$.
The estimate from above then follows by 
\begin{align}\label{eq:2}
\sigma_m(x)_\infty=\sup_{j\ge m+1}x^*_j=x^*_{m+1}
\le \biggl((m+1)^{-1}\sum_{j=1}^{m+1}(x_j^*)^p\biggr)^{1/p}
\le (m+1)^{-1/p} \|x\|_p.
\end{align}
The lower estimate is supplied by taking
\begin{equation}\label{eq:3.5}
x=(m+1)^{-1/p}\sum_{j=1}^{m+1}e_j,
\end{equation}
where $\{e_j\}_{j=1}^\infty$ are the canonical unit vectors.

For general $q$, the estimate from above in \eqref{eq:1} may be obtained from \eqref{eq:2}
and H\"older's inequality
\begin{equation}\label{eq:6}
\|x\|_q\le \|x\|_p^\theta\cdot \|x\|_\infty^{1-\theta}, \quad \text{where}\quad
\frac1q=\frac\theta p.
\end{equation}
The estimate from below follows for all $q$'s by simple modification of \eqref{eq:3.5}.\\
The discussion above exhibits two effects.
\begin{enumerate}
\item[(i)] Best $m$-term approximation works particularly well, when $1/p-1/q$ is large, i.e.
if $p<1$ and $q=\infty$.
\item[(ii)] The elements used in the estimate from below (and hence the elements, where the best $m$-term
approximation performs at worse) enjoy a very special structure.
\end{enumerate}
Therefore, there is a reasonable hope, that the best $m$-term approximation could behave better,
when considered in a certain average case. But first we point out two different interesting points of view
on the subject.

\subsection{Connection to compressed sensing}

The interest in $\ell_p$ spaces (and especially in their finite-dimensional counterparts $\ell_p^n$) 
with $0<p<1$ was recently stimulated by the impressive success of the
novel and vastly growing area of \emph{compressed sensing} as introduced in \cite{CRT1,CRT2,CT,Do}.
Without going much into the details, we only note, that the techniques of compressed sensing allow to reconstruct
a vector from an incomplete set of measurements utilizing the prior knowledge, that it is sparse, i.e. $\|x\|_{0}$
is small. Furthermore, this approach may be applied \cite{CDD} also to vectors, which are \emph{compressible}, i.e. $\|x\|_p$
is small for (preferably small) $0<p<1$. Indeed, \eqref{eq:1} tells us, that such a vector $x$ may be very
well approximated by sparse vectors. 
We point to \cite{Ca,Fo,FoRa,Ra} for the current state of the art of this field and for further references.

This leads in a very natural way to a question, which stands in the background of this paper, namely:

\vskip10pt
\centerline{\emph{How does a typical vector of the $\ell^n_p$ unit ball look like?}}
\vskip10pt
or, posed in an exact way:
\vskip10pt
\emph{Let $\mu$ be a probability measure on the unit ball of $\ell_p^n$.
What is the mean value of $\sigma_m(x)_q$ with respect to this measure?}
\vskip10pt
Of course, the choice of $\mu$ plays a crucial role. There are several standard probability measures,
which are connected to the unit ball of $\ell_p^n$ in a natural way, namely (cf. Definitions \ref{dfn2} and \ref{dfn3})
\begin{itemize}
\item[(i)] the normalized Lebesgue measure,
\item[(ii)] the $n-1$ dimensional Hausdorff measure restricted to the surface of the unit ball of $\ell_p^n$
and correspondingly normalized,
\item[(iii)] the so-called normalized cone measure.
\end{itemize}

Unfortunately, it turns out, that all these three measures are ``bad'' -- a typical vector
with respect to any of them does not involve much structure and corresponds rather to noise then signal
(in the sense described below). Therefore, we are looking for a new type of measures 
(cf. Definition \ref{def14}), which would behave better from this point of view.

\subsection{Random models of noise and signals}

Random vectors play an important role in the area of signal processing. For example, if $n\in \N$ is a natural number,
$\omega=(\omega_1,\dots,\omega_n)$ is a vector of independent Gaussian variables and $\varepsilon>0$ is a real number, 
then $\varepsilon \omega$ is a classical model of noise, namely the \emph{white noise}.
This model is used in the theory but also in the real life applications of signal processing.

The random generation of a structured signal seems to be a more complicated task.
Probably the most common random model to generate sparse vectors, cf. \cite{BSFMD,CGI,GS,PKLH},
is the so-called \emph{Bernoulli-Gaussian model}.
Let again $n\in \N$ be a natural number and $\varepsilon>0$ be a real number. Also
$\omega=(\omega_1,\dots,\omega_n)$ stands for a vector of independent Gaussian variables. Furthermore,
let $0<p<1$ be a real number and let $\varrho=(\varrho_1,\dots,\varrho_n)$ be a vector of 
independent Bernoulli variables defined as
$$
\varrho_i=\begin{cases}1,&\text{with probability }p,\\
0,&\text{with probability }1-p.
\end{cases}
$$
The components of the random \emph{Bernoulli-Gaussian vector} $x=(x_1,\dots,x_n)$ are then defined through
\begin{equation}\label{eq:BG}
x_i=\varepsilon\varrho_i\cdot\omega_i,\quad i=1,\dots,n.
\end{equation}
Obviously, the average number of non-zero components of $x$ is $k:=pn.$ Unfortunately, if $k$ is
much smaller than $n$, then the concentration of the number of non-zero components of $x$ around $k$
is not very strong. This becomes better, if $k$ gets larger. But in that case, the model \eqref{eq:BG}
resembles more and more the model of white noise. In some sense, \eqref{eq:BG} represents rather
a randomly filtered white noise then a structured signal. It is one of the main aims of this paper
to find a new measure, such that a random vector with respect to this measure would show a nearly sparse
structure without the need of random filtering.

\subsection{Unit sphere}

Let us describe the situation in the most prominent case, when $p=2$, $m=0$ and $\mu=\mu_2$ is 
the normalized surface measure on the unit sphere $\S^{n-1}$ of $\ell_2^n$. Furthermore, we denote by $\gamma_n$
the standard Gaussian measure on $\R^n$ with the density
$$
\frac{1}{(2\pi)^{n/2}}e^{-\|x\|^2_2/2},\quad x\in\R^n.
$$
We use polar coordinates to calculate
\begin{align}
\notag\int_{\R^n}\max_{j=1,\dots,n}|x_j|\,d\gamma_n(x)&=\frac{1}{(2\pi)^{n/2}}\int_{\R^n}\max_{j=1,\dots,n}|x_j|\cdot e^{-\|x\|_2^2/2}dx\\
\notag&=\frac{\Omega_n}{(2\pi)^{n/2}} \int_0^\infty r^{n-1} \int_{\S^{n-1}}\max_{j=1,\dots,n}|rx_j|e^{-\|rx\|_2^2/2} d\mu_2(x)\, dr\\
\label{eq:cal1}&=\frac{\Omega_n}{(2\pi)^{n/2}} \int_0^\infty r^{n}e^{-r^2/2}dr\cdot \int_{\S^{n-1}}\max_{j=1,\dots,n}|x_j|d\mu_2(x)\\
\notag&=\frac{\Omega_n}{(2\pi)^{n/2}} \int_0^\infty r^{n}e^{-r^2/2}dr\cdot \int_{\S^{n-1}}\sigma_0(x)_\infty d\mu_2(x),
\end{align}
where $\Omega_n$ denotes the area of $\S^{n-1}$. This formula connects the expected value of $\sigma_0(x)_\infty$
with the expected value of maximum of $n$ independent Gaussian variables. 
Using that this quantity is known to be equivalent to $\sqrt{\log(n+1)}$, cf. \cite[(3.14)]{LT},
$$
\int_0^\infty r^{n}e^{-r^2/2}dr=2^{(n-1)/2}\Gamma((n+1)/2)\quad \text{and}\quad \Omega_{n}=\frac{2\pi^{n/2}}{\Gamma(n/2)},
$$ 
one obtains
\begin{equation}\label{eq:cal2}
\int_{\S^{n-1}}\sigma_0(x)_\infty d\mu_2(x)\approx\sqrt{\frac{\log(n+1)}{n}},\quad n\in \N.
\end{equation}
Several comments on \eqref{eq:cal1} and \eqref{eq:cal2} are necessary.
\begin{itemize}
\item[(i)] Quantities similar to the left-hand side of \eqref{eq:cal2} have been used in the study of geometry of Banach spaces
and local theory of Banach spaces since many years and are treated in detail in the work of Milman \cite{FLM,M,MS}. 
Especially, if $\|\cdot\|_K$ is a norm in $\R^n$
and $K:=\{x\in\R^n:\|x\|_K\le 1\}$ denotes the corresponding unit ball, then the quantity
$$
A_K=\int_{\S^{n-1}}\|x\|_K d\mu_2(x)
$$
(and the closely connected median $M_K$ of $\|x\|_K$ over $\S^{n-1}$) 
plays a crucial role in the Dvoretzky  theorem \cite{Dv,F,M} and, in general, in the study of Euclidean sections of $K$, cf. \cite[Section 5]{MS}.
Furthermore, it is known that the case of $K=[-1,1]^n$, when
$$
A_K=\int_{\S^{n-1}}\max_{j=1,\dots,n}|x_j| d\mu_2(x)=\int_{\S^{n-1}}\sigma_0(x)_\infty d\mu_2(x),
$$
is extremal, cf. \cite{M}.
\item[(ii)] The connection between the estimated value of a maximum of independent Gaussian variables
and the estimated value of the largest coordinate of a random vector on $\S^{n-1}$ is given just by integration in polar coordinates
and is one of the standard techniques in the local theory of Banach spaces.
Due to the result of \cite{SZ},
this holds true also for other values of $p$, even for $p<1$, with Gaussian variables replaced by variables with
the density $c_pe^{-|t|^p}$. This approach is nowadays
classical in the study of the geometry and concentration of measure phenomenon on the $\ell_p^n$-balls, 
cf. \cite{ABP,BP,BCN,BGMN,N,NR,RR}.
\item[(iii)] For every $x\in\S^{n-1}$ we obtain easily that $\displaystyle \max_{j=1,\dots,n}|x_j|\ge 
\Bigl(\frac{1}{n}\sum_{j=1}^n x_j^2\Bigr)^{1/2}=1/\sqrt{n}$. Estimate \eqref{eq:cal2} shows that
the average value of $\displaystyle \max_{j=1,\dots,n}|x_j|$ over ${\S}^{n-1}$ is asymptotically
larger only by a logarithmic factor. The detailed study of the concentration of $\displaystyle \max_{j=1,\dots,n}|x_j|$
around its estimated value (or its mean value) is known as \emph{concentration of measure phenomena} \cite{Le,LT,MS}
and gives more accurate information then the one included in \eqref{eq:cal2}.
As our main interest lies in estimates of \emph{average best $m$-term widths}, cf. Definition \ref{dfn1},
we do not investigate the concentration properties in this paper and leave this subject to further research.
\item[(iv)] The calculation \eqref{eq:cal1} is based on the use of polar coordinates. For $p\not=2$, the normalized cone measure is exactly that
measure, for which a similar formula holds, cf. \eqref{eq:4}. The estimates for $n-1$ dimensional surface measure
are later obtained using its density with respect to the cone measure, cf. Lemma \ref{lem7}.
\item[(v)] As we want to keep the paper self-contained as much as possible and to make it readable also for readers without (almost)
any stochastic background, we prefer to use simple and direct techniques. For example we use rather the simple estimates in Lemma \ref{lem3},
than any of their sophisticated improvements available in literature.
\item[(vi)] The connection to random Gaussian variables explains, why a random point of $\S^{n-1}$ is sometimes referred to as 
\emph{white (or Gaussian) noise}. It is usually not associated with any reasonable (i.e. structured) signal, 
rather it represents a good model for random noise.
\end{itemize}

\subsection{Basic Definitions and Main Results}

\subsubsection{Definition of average best $m$-term widths}

After describing the context of our work we shall now present the definition of the so-called
\emph{average best $m$-term widths}, which are the main subject of our study.

First, we observe, that
$$
\sigma_{m}((x_1,\dots,x_n))_q=\sigma_m((\varepsilon_1 x_1,\dots,\varepsilon_n x_n))_q=
\sigma_m((|x_1|,\dots,|x_n|))_q
$$
holds for every $x\in\R^n$ and $\varepsilon\in\{-1,+1\}^n$. Also all the measures, which we shall consider,
are invariant under any of the mappings
$$
(x_1,\dots,x_n)\to (\varepsilon_1 x_1,\dots,\varepsilon_n x_n),\quad \varepsilon\in\{-1,+1\}^n
$$
and therefore we restrict our attention only to $\R^n_+$ in the following definition.
\begin{dfn}\label{dfn1}
Let $0<p\le q\le\infty$ and let $n\ge 2$ and $0\le m\le n-1$ be natural numbers.
\begin{enumerate}
\item[(i)]We set
$$
\D = \begin{cases} \{(t_1,\dots,t_n)\in \R_+^n: \sum_{j=1}^n{t^p_j}=1\},&\quad p<\infty,\\
\{(t_1,\dots,t_n)\in \R_+^n: \max_{j=1,\dots,n} t_j=1\}, &\quad p=\infty.
\end{cases}
$$
\item[(ii)] Let $\mu$ be a Borel probability measure on $\D$. Then
$$
\sigma_m^{p,q}(\mu)=\int_{\D} \sigma_m(x)_qd\mu(x)
$$
is called {\it average surface best $m$-term width of $id:\ell_p^n\to\ell_q^n$
with respect to $\mu$}.
\item[(iii)] Let $\nu$ be a Borel probability measure on $[0,1]\cdot\D$. Then
$$
\sigma_m^{p,q}(\nu)=\int_{[0,1]\cdot\D} \sigma_m(x)_qd\nu(x)
$$
is called {\it average volume best $m$-term width of $id:\ell_p^n\to\ell_q^n$
with respect to $\nu$}.
\end{enumerate}
\end{dfn}
Let us observe, that the estimates
$$
 \sigma_{m}^{p,q}(\mu)\le \sigma_m^{p,q}\quad\text{and}\quad
 \sigma_{m}^{p,q}(\nu)\le \sigma_m^{p,q}
$$
follow trivially by Definition \ref{dfn1}.
Furthermore, the mapping $x\to \sigma_m(x)_q$ is continuous and, therefore, measurable with respect to the Borel measure $\mu$. 

\subsubsection{Main results}

After introducing new notion of average best $m$-term width in Definition \ref{dfn1}, we
study its behavior for the measures on $\D$, which are widely used in literature.
A prominent role among them is played by the so-called \emph{normalized cone measure} given by
$$
 {\mu_p}({\mathcal A})=
 \frac{\lambda([0,1]\cdot {\mathcal A})}{\lambda([0,1]\cdot \D)},\quad {\mathcal A}\subset\D.
$$
In Theorem \ref{thm6} and Proposition \ref{prop12} we provide basic estimates of $\sigma_{m}^{p,q}(\mu_p)$
for $q=\infty$ and $q<\infty$, respectively. Surprisingly enough, it turns out that \eqref{eq:cal2} has its direct counterpart 
for all $0<p<\infty$. This means (as described above), that the coordinates of a ``typical'' element of the 
surface of the $\ell_p^n$ unit ball are well concentrated around the value $n^{-1/p}$. So, roughly speaking, it is only 
$\ell_p$-normalized noise.

Another well known probability measure on $\D$ is the \emph{normalized surface measure} $\varrho_p$, cf. Definition \ref{dfn3}.
We calculate in Lemma \ref{lem7} the density of $\varrho_p$ with respect to $\mu_p$ to be equal to
$$
\frac{d\varrho_p} {d\mu_p}(x)= c^{-1}_{p,n} \biggl(\sum_{i=1}^{n}x_i^{2p-2}\biggr)^{1/2},
$$
where
$$
c_{p,n}=\int_{\D}\biggl(\sum_{i=1}^n x_i^{2p-2}\biggr)^{1/2}d \mu_p(x)
$$
is the normalizing constant. This result (which is a generalization of the work of Naor and Romik \cite{NR} to the non-convex case $0<p<1$)
might be of independent interest for the study of the geometry of $\ell_p^n$ spheres.
One observes immediately, that if $p<1$ and one or more coordinates
of $x_i$ are going to zero, then this density has a polynomial singularity and, therefore, gives more
weight to areas closed to coordinate hyperplanes.

We then obtain in Theorem \ref{thm9}
an estimate of $\sigma_{0}^{p,\infty}(\varrho_p)$ from above. Although the measure $\varrho_p$
concentrates around coordinate hyperplanes, it turns out,
that the estimate from above of $\sigma_0^{p,\infty}(\mu_p)$ as obtained in Theorem \ref{thm6} and the estimate of Theorem \ref{thm9}
differ only in the constants involved.

The last part of this paper is devoted to the search of a new probability measure on $\D$, which would ``promote sparsity'' in the sense, that
the mean value of $\sigma_m(x)_q$ decays rapidly with $m$. One possible candidate is presented in Definition \ref{def14} by introducing a new class of measures
$\theta_{p,\beta}$, which are given by their density with respect to the cone measure $\mu_p$
\begin{equation*}
\frac{d \theta_{p,\beta}}{d\mu_p}(x)=c^{-1}_{p,\beta}\cdot \prod_{i=1}^n x_i^{\beta}, \quad x\in\D,
\end{equation*}
where $c_{p,\beta}$ is a normalising constant. We refer also to Remark \ref{rem4} for an equivalent characterisation.

We show, that for an appropriate choice of $\beta$, namely $\beta=p/n-1$,
the estimated value of the $m$-th largest coefficient of elements of the $\ell_p^n$-unit sphere decays exponentially with $m$.
Namely, Theorem \ref{thm17} provides estimates of $\sigma_{m-1}^{p,\infty}(\theta_{p,p/n-1})$, which at the end imply that
\begin{align}\label{eq:intr}
\frac{C_p^1}{\left(\frac{1}{p}+1\right)^m}
\le \liminf_{n\to\infty} \sigma_{m-1}^{p,\infty}(\theta_{p,p/n-1})
\le \limsup_{n\to\infty} \sigma_{m-1}^{p,\infty}(\theta_{p,p/n-1}) 
\le \frac{C_p^2}{\left(\frac{1}{p}+1\right)^m}
\end{align}
for two positive real numbers $C_p^1$ and $C_p^2$, which depend only on $p$.

This result (which is also simulated numerically in the very last section of this paper) 
is in a certain way independent of $n$. This gives a hope, that one could apply this approach also to
the infinite-dimensional spaces $\ell_p$ or, using a suitable discretization technique (like wavelet decomposition),
also to some function spaces. This remains a subject of our further research.

Of course, the class $\theta_{p,\beta}$ provides only one example of measures with rapid decay of their average best $m$-term widths.
We leave also the detailed study of other measures with such properties open to future work.

{\bf Note added in the proof:} Let us comment on the relation of our work with recent papers of Cevher \cite{Cev} and
Gribonval, Cevher, and Davis \cite{GCD}. Cevher uses in \cite{Cev} the concept of \emph{Order Statistics} \cite{DN}
to identify the probability distributions, whose independent and identically distributed (i.i.d.) realizations
result typically in $p$-compressible signals, i.e.
$$
x^*_{i}\le C\, R\cdot i^{-1/p}.
$$
Our approach here is a bit different and more connected to the geometry of $\ell_p^n$ spaces. In accordance
with \cite{SZ}, this leads to the study of $\ell_p^n$-\emph{normalized} vectors with i.i.d. components.
This again allows us to better distinguish between the norm of such a vector (i.e. its \emph{size} or \emph{energy})
and its direction (i.e. its \emph{structure}).

The approach of the recent preprint \cite{GCD} (which was submitted during the review process of this work) 
comes much closer to ours.
Their Definition 1 of ``Compressible priors'' introduces the quantity called \emph{relative best $m$-term approximation error} as
$$
\bar\sigma_m(x)_q=\frac{\sigma_m(x)_q}{\|x\|_q},\quad x\in\R_+^n.
$$
The asymptotic behavior of this quantity for $x=(x_1,\dots,x_n)$ being a vector with i.i.d. components and
$\liminf_{n\to\infty}\frac{m_n}{n}\ge\kappa\in(0,1)$ is then used to define $q$-compressible probability distribution functions.
In contrary to \cite{GCD}, we consider $\ell_q$ approximation of $\ell_p$ normalized vectors and therefore our
widths depend on two integrability parameters $p$ and $q$.
Furthermore, we do not pose any restrictions on the ratio $m/n$ to any specific regime and consider 
the average best $m$-term widths
$\sigma^{p,q}_m(\mu)$ for all $0\le m\le n-1.$ In the only case, when we speak about asymptotics (i.e. \eqref{eq:final1} 
of Theorem \ref{thm17}), we suppose $m$ to be constant and $n$ growing to infinity.
Furthermore, Theorem 1 of \cite{GCD} shows that all distributions with bounded fourth moment
do not fit into their scheme and do not ``promote sparsity''. As we are interested in distributions, which are connected to the geometry
of $\ell_p^n$-balls (i.e. generalized Gaussian distribution and generalized Gamma distribution),
it is exactly that reason why we change the parameters of the distribution $\theta_{p,\beta}$ in dependence of $n$.
Although quite inconvenient from the mathematical point of view, it is not really clear if this presents a serious obstacle
for application of our approach. But the investigation of this goes beyond the scope of this work.

\subsubsection{Structure of the paper}

The paper is structured as follows. The rest of Section 1 gives some notation used throughout the paper.
Sections 2 and 3 provide estimates of this quantity with respect to the cone and surface measure, respectively.
In Section 4, we study a new type of measures on the unit ball of $\ell_p^n$. We show, that the typical
element with respect to those measures behaves in a completely different way compared to the situations
discussed before. Those results are illustrated by the numerical experiments described in Section 5.

\subsection{Notation}

We denote by $\R$ the set of real numbers, by $\R_+:=[0,\infty)$ the set of nonnegative real numbers
and by $\R^n$ and $\R_+^n$ their $n$-fold tensor products. The components of $x\in\R^n$ are denoted
by $x_1,\dots,x_n$. The symbol $\lambda$ stands for the Lebesgue measure on $\R^n$ and 
${\mathcal H}$ for the $n-1$ dimensional Hausdorff measure in $\R^n$. If $A\subset \R^n$ and $I\subset \R$ is an interval, 
we write $I\cdot A:=\{tx:t\in I, x\in A\}.$

We shall use very often the \emph{Gamma function}, defined by
\begin{equation}\label{eq:Gamma}
\Gamma(s):=\int_0^\infty t^{s-1}e^{-t}dt,\quad s>0.
\end{equation}
In one case, we shall use also the \emph{Beta function}
\begin{equation}\label{eq:Beta}
B(p,q):=\int_0^1 t^{p-1}(1-t)^{q-1}dt=\frac{\Gamma(p)\Gamma(q)}{\Gamma(p+q)},\quad p,q>0
\end{equation}
and the \emph{digamma function}
$$
\Psi(s):=\frac{d}{ds} \log \Gamma(s)=\frac{\Gamma'(s)}{\Gamma(s)},\quad s>0.
$$
We recommend \cite[Chapter 6]{AS} as a standard reference for both basic and more advanced properties of these functions. 
We shall need the Stirling's approximation formula 
(which was implicitly used already in \eqref{eq:cal2}) in its most simple form
\begin{equation}\label{eq:Stirl}
\Gamma(x)=\sqrt{\frac{2\pi}{x}}\left(\frac{x}{e}\right)^x\left(1+{\mathcal O}\left(\frac{1}{x}\right)\right),\quad x>0.
\end{equation}

If $a=\{a_j\}_{j=1}^\infty$ and $b=\{b_j\}_{j=1}^\infty$ are real sequences, 
then $a_j\lesssim b_j$ means, that there is an absolute constant $C>0$, such that $a_j\le C\,b_j$ for all $j=1,2,\dots$.
Similar convention is used for $a_j \gtrsim b_j$ and $a_j\approx b_j.$ The capital letter $C$ with indices (i.e. $C_p$)
denotes a positive real number depending only on the highlighted parameters and their meaning can change from one
occurrence to another. If, for any reason, we shall need to distinguish between several numbers of this type, we
shall write for example $C_p^1$ and $C_p^2$ as already done in \eqref{eq:intr}.


\section{Normalized cone measure}

In this section, we study the
average best $m$-term widths as introduced in Definition \ref{dfn1} for the most important measure 
(the so-called cone measure) on $\D$, which is well studied in the
literature within the geometry of $\ell_p^n$ spaces, cf. \cite{NR,BCN,N, BGMN}.
Essentially, we recover in Theorem \ref{thm6} an analogue of the estimate \eqref{eq:cal2} for all $0<p<\infty.$

\begin{dfn}\label{dfn2}
Let $0<p\le\infty$ and $n \ge 2$. Then 
$$
 {\mu_p}({\mathcal A})=
 \frac{\lambda([0,1]\cdot {\mathcal A})}{\lambda([0,1]\cdot \D)},\quad {\mathcal A}\subset\D
$$
is the normalized {\it cone measure} on $\D$.
\end{dfn}
If $\nu_p$ denotes the $p$-normalized Lebesgue measure, i.e.
$$
\nu_p(A)=\frac{\lambda(A)}{\lambda([0,1]\cdot \D)},\quad A\subset \R_+^n,
$$
then the connection between $\nu_p$ and $\mu_p$ is given by
\begin{equation}\label{eq:3}
\nu_p(A)=n\int_0^\infty r^{n-1}\mu_p\biggr(\frac{\{x\in A:\|x\|_p=r\}}{r}\biggl)dr.
\end{equation}
The proof of \eqref{eq:3} follows directly for sets of the type $[a,b]\cdot {\mathcal A}$ with $0<a<b<\infty$
and ${\mathcal A}\subset \D$ and is then finished by standard approximation arguments.
The formula \eqref{eq:3} may be generalized to the so-called {\it polar decomposition identity}, cf. \cite{BCN},
\begin{equation}\label{eq:4}
\frac{\displaystyle\int_{\R_+^n}f(x)d\lambda(x)}{\lambda([0,1]\cdot\D)}=
n\int_0^\infty r^{n-1}\int_{\D}f(rx)d\mu_p(x)dr,
\end{equation}
which holds for every $f\in L_1(\R_+^n)$.

The formula \eqref{eq:4} allows to transfer immediately the results for the average surface best $m$-term
approximation with respect to $\mu_p$ to the average volume approximation with respect to $\nu_p$.

\begin{prop}\label{prop10} 
The identity
$$
 \sigma_{m}^{p,q}(\nu_p)=\sigma_{m}^{p,q}(\mu_p)\cdot \frac{n}{n+1}
$$
holds for all $0<p\le q\le \infty$, all $n\ge 2$ and all $0\le m\le n-1$.
\end{prop}
\begin{proof}
We plug the function
$$
f(x)=\sigma_m(x)_q\cdot\chi_{[0,1]\cdot\D}(x)
$$
into \eqref{eq:4} and obtain
\begin{align*}
\frac{\displaystyle\int_{[0,1]\cdot\D}\sigma_m(x)_qd\lambda(x)}{\lambda([0,1]\cdot \D)}&=
\int_{[0,1]\cdot\D}\sigma_m(x)_q d\nu_p(x)\\
&=n\int_0^1 r^{n-1}\int_{\D}\sigma_m(rx)_qd\mu_p(x)dr
=n\int_0^1r^ndr\cdot \sigma_{m}^{p,q}(\mu_p),
\end{align*}
which gives the result.
\end{proof}
Proposition \ref{prop10} shows, that the ratio between 
approximation with respect to $\mu_p$ and $\nu_p$
is equal to $1+1/n$. This justifies our interest in measures on $\D$.
Furthermore, it shows that the quantities
$\sigma_{m}^{p,q}(\nu_p)$ and $\sigma_{m}^{p,q}(\mu_p)$ behave asymptotically (i.e. for $n\to\infty$)
very similarly.

Let $p=2$ and let $\omega_1,\dots,\omega_n$ be independent normally distributed Gaussian random variables.
Then
$$
\varrho_2({\mathcal A})=\mu_2({\mathcal A})=
{\mathbb P}\Biggl(\frac{(|\omega_1|,\dots,|\omega_n|)}{\bigl(\sum_{j=1}^n{\omega_j^2}\bigr)^{1/2}}\in 
{\mathcal A}\Biggr),\qquad {\mathcal A}\subset \Delta_2^n.
$$
As noted in \cite{SZ}, this relation may be generalized to all values of $p$ with $0<p<\infty$.
Let $\omega_1,\dots,\omega_n$ be independent random variables on $\R_+$ each with density
$$
  c_p e^{-t^p},\quad t\ge 0
$$
with respect to the Lebesgue measure, where $c_p=\frac{p}{\Gamma(1/p)}=\frac{1}{\Gamma(1/p+1)}$.

Then, cf. \cite[Lemma 1]{SZ},
\begin{equation}\label{eq:gener1}
\mu_p({\mathcal A})=
{\mathbb P}\Biggl(\frac{(\omega_1,\dots,\omega_n)}{\bigl(\sum_{j=1}^n{\omega_j^p}\bigr)^{1/p}}\in 
{\mathcal A}\Biggr), \qquad {\mathcal A}\subset\D.
\end{equation}
We shall fix $\omega_1,\dots,\omega_n$ to the end of this paper. Also the symbols 
${\mathbb E}$ and ${\mathbb P}$ are always taken with respect to these variables.

\subsection{The case $q=\infty$}
In this section we deal with uniform approximation, i.e. with the case $q=\infty$.
To be able to imitate the calculation \eqref{eq:cal1}, we shall need several tools,
which are subject of Lemmas \ref{lem1}, \ref{lem3} and \ref{lem2}.
Our main result of this section (Theorem \ref{thm6}) then provides the estimate 
of $\sigma_m^{p,\infty}(\mu_p)$ from above for all $m$ with $0\le m\le n-1$. Furthermore,
it is shown that in the range $0\le m\le \varepsilon_p n$ this estimate is also optimal.

\begin{lem}\label{lem1} 
Let $0<p<\infty$ and let $n\ge 2$ and $1\le m\le n$ be natural numbers. Then
$$
\int_{\D}x_m^*d\mu_p(x)=
\frac{\Gamma(n/p)}{\Gamma(n/p+1/p)}\cdot {\mathbb E\,}x_m^*.
$$
Furthermore, there are two positive real numbers $C_p^1$ and $C_p^2$ depending only on $p$, such that
$$
C_p^1\cdot \frac{{\mathbb E\,}x_m^*}{n^{1/p}}\le \int_{\D}x_m^*d\mu_p(x)\le C_p^2\cdot \frac{{\mathbb E\,}x_m^*}{n^{1/p}}.
$$
\end{lem}
\begin{proof}
We put $f(x)=x_m^*e^{-x_1^p-\dots-x_n^p}$ and use the polar decomposition identity \eqref{eq:4}
\begin{align*}
\frac{\displaystyle\int_{\R_+^n}x_m^*e^{-x_1^p-\dots-x_n^p}d\lambda(x)}{\lambda([0,1]\cdot \D)}
&=n\int_0^\infty r^{n-1}\int_{\D} (rx_m^*)\cdot e^{-(rx_1)^p-\dots-(rx_n)^p}d\mu_p(x)dr\\
&=n\int_0^\infty r^{n-1}\cdot re^{-r^p}dr\int_{\D} x_m^*d\mu_p(x)
\end{align*}
or, equivalently,
\begin{equation}\label{eq:unkn}
\int_{\D}x_m^*d\mu_p(x)=\frac{\displaystyle\int_{\R_+^n}x_m^*e^{-x_1^p-\dots-x_n^p}d\lambda(x)}
{\lambda([0,1]\cdot \D)\cdot n\int_{0}^\infty r^ne^{-r^p}dr}.
\end{equation}
The identity
$$
\int_0^\infty r^ne^{-r^p}dr=\frac {\Gamma(n/p+1/p)}{p},
$$
follows by a simple substitution. Furthermore, we shall need the classical formula of Dirichlet
for the volume of the unit ball $B_{\ell^n_p}$ of $\ell_p^n$, cf. \cite[p. 157]{E},
$$
\lambda([0,1]\cdot \D)=\frac{\lambda(B_{\ell_p^n})}{2^n}=\frac{\Gamma(1/p+1)^n}{\Gamma(n/p+1)}.
$$
This allows us to reformulate \eqref{eq:unkn} as
$$
\int_{\D}x_m^*d\mu_p(x)=\frac{\Gamma(n/p+1)\,{\mathbb E\,} x_m^*}
{c_p^n\cdot n/p\cdot\Gamma(n/p+1/p)\Gamma(1/p+1)^n}=
\frac{\Gamma(n/p)\,{\mathbb E\,}x_m^*}{\Gamma(n/p+1/p)}.
$$
Finally, we use Stirling's formula \eqref{eq:Stirl} to estimate
\begin{align*}
\frac{n^{1/p}\cdot \Gamma(n/p) }{\Gamma(n/p+1/p)}
\le C_p^1 \frac{n^{1/p}(n/p)^{n/p-1/2}}{(n/p+1/p)^{n/p+1/p-1/2}}
\le C_p^2\biggl(\frac{n}{n+1}\biggr)^{n/p+1/p-1/2}\le C_p^3
\end{align*}
and similarly for the estimate from below.
\end{proof}

\begin{lem}\label{lem3}
Let $\alpha\in \R$ and $\delta>0$. Then
$$
 \int_\delta^\infty u^{\alpha}e^{-u}du\le
 \delta^\alpha e^{-\delta}\cdot \begin{cases}
 1,\qquad &\text{if}\quad \alpha\le 0,\\
 \frac{1}{1-\alpha/\delta},\qquad &\text{if}\quad \alpha>0\quad \text{and}\quad \frac{\alpha}{\delta}<1,\\
 \bigl(\frac{\alpha}{\delta}\bigr)^{\alpha}\cdot
 \frac{\alpha/\delta}{1-\delta/\alpha},\qquad &\text{if}\quad \alpha>0\quad \text{and}\quad \frac{\alpha}{\delta}>1.\\
 \end{cases}
$$
\end{lem}
\begin{proof}
If $\alpha\le 0$, we may estimate
$$
 \int_\delta^\infty u^{\alpha}e^{-u}du\le \delta^{\alpha}
 \int_\delta^\infty e^{-u}du=\delta^{\alpha}e^{-\delta}.
$$
If $0<\alpha\le 1$, we use partial integration and obtain
$$
 \int_\delta^\infty u^{\alpha}e^{-u}du=
 \delta^\alpha e^{-\delta}+\alpha\int_\delta^\infty u^{\alpha-1}e^{-u}du
 \le \delta^\alpha e^{-\delta}(1+\alpha\delta^{-1}).
$$
This is smaller than 
$$
\delta^\alpha e^{-\delta} (1+\frac{\alpha}{\delta}+\frac{\alpha^2}{\delta^2}+\dots)=\delta^\alpha
e^{-\delta}\cdot \frac{1}{1- \alpha/\delta}
$$
if $\alpha/\delta<1$ and smaller than
$$
 \delta^\alpha e^{-\delta}\frac{\alpha}{\delta}(1+\frac{\delta}{\alpha}+\frac{\delta^2}{\alpha^2}+\dots)=
 \delta^\alpha e^{-\delta}\frac{\alpha}{\delta}\cdot \frac{1}{1-\delta/\alpha}.
$$
if $\alpha/\delta>1$.

If $k-1< \alpha\le k$ for some $k\in \N$, we iterate the partial integration and arrive at
\begin{align*}
 \int_\delta^\infty u^{\alpha}e^{-u}du&\le
 \delta^\alpha e^{-\delta}(1+\alpha\delta^{-1}+\alpha(\alpha-1)\delta^{-2}+
 \dots+\alpha(\alpha-1)\dots(\alpha-k+1)\delta^{-k})\\
 &\le \delta^\alpha e^{-\delta}
 (1+\frac{\alpha}{\delta}+\frac{\alpha^2}{\delta^2}+\dots+\frac{\alpha^k}{\delta^k})\\
 &\le \delta^\alpha e^{-\delta}\begin{cases}
 \frac{1}{1-\alpha/\delta},\quad \text{if}\quad \alpha/\delta<1,\\
 \bigl(\frac{\alpha}{\delta}\bigr)^{\alpha+1}\frac{1}{1-\delta/\alpha},\quad \text{if}\quad \alpha/\delta>1.\\
 \end{cases}
\end{align*}
\end{proof}

\begin{lem}\label{lem2}
Let $0< p<\infty$. Then there is a positive real number $C_p$, such that
$$
 {\mathbb E\,}x_m^*\le C_p \log^{1/p}\Bigl(\frac{en}{m}\Bigr)
$$
for all $1\le m\le n$.
\end{lem}
\begin{proof}
We estimate
\begin{align}
\notag{\mathbb E}\, x_m^*&=\int_0^\infty {\mathbb P}(\omega_m^*>t)dt
=\delta+\int_\delta^\infty {\mathbb P}(\omega_m^*>t)dt\\
&\label{eq:7} \le \delta+\binom{n}{m}\int_{\delta}^\infty {\mathbb P}(\omega_1>t,\omega_2>t,\dots,\omega_m>t)dt\\
&\notag= \delta+\binom{n}{m}\int_{\delta}^\infty {\mathbb P}(\omega_1>t)^mdt.
\end{align}
The parameter $\delta>\max(1,3(1/p-1))^{1/p}$ is to be chosen later on.
We substitute $v=u^p$ and obtain
\begin{align*}
{\mathbb P}(\omega_1>t)=c_p \int_t^\infty e^{-u^p}du
=\frac{c_p}{p}\int^\infty_{t^p}v^{1/p-1}e^{-v}dv.
\end{align*}
Using the first two estimates of Lemma \ref{lem3} (recall that $t^p\ge \delta^p>\max(1,3(1/p-1))$), we arrive at
$$
{\mathbb P}(\omega_1>t)\le C_pt^{1-p}e^{-t^p},
$$
where $C_p$ depends only on $p$. We plug this estimate into \eqref{eq:7} and obtain
\begin{equation}\label{eq:delta}
{\mathbb E\,}x_m^*\le \delta+\binom{n}{m}(C_p)^m
\int_\delta^\infty t^{m(1-p)}e^{-mt^p}dt.
\end{equation}
If $p\ge 1$, then 
\begin{align*}
\int_\delta^\infty t^{m(1-p)}e^{-mt^p}dt\le 
\delta^{m(1-p)}\int_\delta^\infty e^{-mt^p}dt
\le \delta^{m(1-p)} \int_{m\delta^{p}}^\infty e^{-u}u^{1/p-1}du
\le e^{-m\delta^p}.
\end{align*}
Altogether, we obtain
$$
{\mathbb E\,}x_m^*\le \delta + \binom{n}{m}(C_p)^m e^{-m\delta^p}.
$$
Using $\binom{n}{m}\le (\frac{en}{m})^m$ and choosing $\delta=C_p'\ln(\frac{en}{m})^{1/p}$
finishes the proof.

If $p<1$, we use again the second estimate of Lemma \ref{lem3}
\begin{align*}
\int_\delta^\infty t^{m(1-p)}e^{-mt^p}dt
&=\frac{1}{mp}\cdot m^{(1/p-1)(m+1)} \int_{m\delta^p}^\infty
u^{(1/p-1)(m+1)}e^{-u}du\\
&\le \frac{1}{mp} \cdot \delta^{(1-p)(m+1)}e^{-m\delta^p}\cdot\frac{1}{1-\frac{2(1/p-1)}{\delta^p}}
\le C'_p\delta^{(1-p)(m+1)}e^{-m\delta^p}.
\end{align*}
Using \eqref{eq:delta} and $\binom{n}{m}\le (\frac{en}{m})^m$ again, we get
\begin{align*}
{\mathbb E\,}x_1^*&\le \delta+\exp(-m\delta^p+m\ln(en/m)+(1-p)(m+1)\ln \delta + m\ln C_p+\ln C_p')\\
&\le \delta+\exp[-m(\delta^p+C_p\ln(en/m)+2(1-p)\ln \delta)]
\end{align*}
The choice $\delta=C_p'\ln(\frac{en}{m})^{1/p}$ with $C'_p$ large enough ensures, that
$$
\frac {\delta^p}{2}\ge C_p\ln(en/m)\quad\text{and}\quad \frac {\delta^p}{2}\ge 2(1-p)\ln \delta
$$
and finishes the proof.
\end{proof}

The following theorem gives the basic estimates of $\sigma^{p,\infty}_m(\mu_p)$.

\begin{thm} \label{thm6}
Let $0< p\le \infty$ and let $n\ge 2$.\\
(i) Let $0\le m\le n-1$. Then
\begin{equation}\label{thm6:1}
\sigma^{p,\infty}_m(\mu_p)\le C_p
\Biggl[\frac{\log\Bigl(\frac{en}{m+1}\Bigr)}{n}\Biggr]^{1/p}.
\end{equation}
(ii) There is a number $0<\varepsilon_p<1$, such that for $0\le m\le \varepsilon_p n$ the following estimate holds
\begin{equation}\label{thm6:2}
\sigma_m^{p,\infty}(\mu_p)\ge C_p \biggl[\frac{\log(\frac{en}{m+1})}{n}\biggr]^{1/p}.
\end{equation}
\end{thm}

\begin{proof}
Lemma \ref{lem1} and Lemma \ref{lem2} imply immediately the first part of the theorem if $p<\infty$.
If $p=\infty$, the proof is trivial.

The proof of the second part is divided into two steps.

\emph{Step 1.} We start first with the case $m=0$.
 
If $p=\infty$, then $x_1^*=1$ for all $x\in\D$ and the proof is trivial. Let us therefore assume, that $p<\infty.$
According to Lemma \ref{lem1}, we have to estimate ${\mathbb E}\, x_1^*$ from below.
This was done in \cite[Lemma 2]{SZ}. We include a slightly different proof for readers convenience.
For every $t_0>0$, it holds
\begin{align*}
{\mathbb E}\, x_1^*
\ge t_0\,{\mathbb P}(x_1^*>t_0)=t_0\, {\mathbb P}(\max_{1\le j\le n} x_j>t_0)
\ge t_0[n{\mathbb P}(x_1>t_0)-\binom{n}{2}{\mathbb P}(x_1>t_0)^2].
\end{align*}
We define $t_0$ by ${\mathbb P}(x_1>t_0)=\frac{1}{n}$ and obtain ${\mathbb E}\, x_1^*\ge t_0/2.$

From the simple estimate
$$
\frac{c_p}{p}\int_{T^p}^{\infty} u^{1/p-1}e^{-u}du
\ge C_p e^{-2T^p},\qquad T>1,
$$
it follows, that there is a positive real number $\gamma_p>0$, such that
$$
{\mathbb P}(x_1>\gamma_p(\log (en))^{1/p})\ge 1/n.
$$
This gives $t_0\ge \gamma_p(\log (en))^{1/p}$ and ${\mathbb E}\, x_1^*\ge C_p (\log (en))^{1/p}.$

\emph{Step 2.} Let $0\le m\le \varepsilon_p n$, where $\varepsilon_p>0$ will be chosen later on.

We shall use the inequality 
\begin{equation}\label{eq:ref}
\frac{1}{m}\sum_{j=1}^m \log^{1/p}\Bigl(\frac{en}{j}\Bigr)\le C_p \log^{1/p}\Bigl(\frac{en}{m}\Bigr),\quad 1\le m \le n,
\end{equation}
which follows by direct calculation for $p=1$, by H\"older's inequality for $1<p<\infty$ and by
replacing the sum by the corresponding integral and integration by parts if $0<p<1.$

We denote 
$$
\|x\|_{(m)}=\frac{1}{m}\sum_{j=1}^m x_j^*.
$$
By Lemma \ref{lem2} and \eqref{eq:ref},
\begin{equation}\label{eq:thm6:1}
\E\, \|x\|_{(m)}=\frac{1}{m}\sum_{j=1}^m \E\, x_j^*\le \frac{C_p}{m}\sum_{j=1}^m \log^{1/p}\Bigl(\frac{en}{j}\Bigr)
\le C^1_p \log^{1/p}\Bigl(\frac{en}{m}\Bigr).
\end{equation}
To estimate $\E\,\|x\|_{(m)}$ from below, we assume that $1\le m\le n$ and that $n/m$ is an integer (otherwise one has to slightly modify
the argument at the cost of the constants involved). We partition the set $\{1,\dots,n\}=A_1\cup\dots\cup A_m$,
where each one of the disjoint sets $A_j$ has $n/m$ elements. Then we have
$$
\|x\|_{(m)}\ge \frac{1}{m}\sum_{j=1}^m\max_{l\in A_j}x_l
$$
and by the first step we obtain
\begin{equation}\label{eq:thm6:2}
\E\,\|x\|_{(m)}\ge \frac{1}{m}\sum_{j=1}^m\E\,\max_{l\in A_j} x_l\ge C^2_p\log^{1/p}\Bigl(\frac{en}{m}\Bigr).
\end{equation}
Let $N_p<1/\varepsilon_p$ be a natural number to be chosen later on. Combining \eqref{eq:thm6:1} with \eqref{eq:thm6:2} gives finally
\begin{align*}
\E\,x_m^*&\ge \frac{1}{N_pm}\sum_{k=m}^{N_pm}\E\, x_k^*\ge \E\,\|x\|_{(N_pm)}-\frac{1}{N_p}\E\,\|x\|_{(m)}\\
&\ge C^2_p\log^{1/p}\Bigl(\frac{en}{N_pm}\Bigr)-\frac{C^1_p}{N_p}\log^{1/p}\Bigl(\frac{en}{m}\Bigr)\\
&=\log^{1/p}\Bigl(\frac{en}{m}\Bigr)\left\{C^2_p\left[1-\frac{\log(N_p)}{\log\Bigl(\frac{en}{m}\Bigr)}\right]^{1/p}-\frac{C^1_p}{N_p}\right\}.
\end{align*}
An appropriate choice of $N_p$ and $\varepsilon_p$ (i.e. $N_p>2^{1/p}C_p^1/C_p^2$ and $\varepsilon_p<\min(1/N_p,e/N_p^2)$) with
$$
C^2_p\left[1-\frac{\log(N_p)}{\log\Bigl(\frac{e}{\varepsilon_p}\Bigr)}\right]^{1/p}-\frac{C^1_p}{N_p}>0
$$
gives the result.
\end{proof}
\begin{rem}
\begin{itemize}
\item[(i)] Theorem \ref{thm6} provides basic estimates of average best $m$-term widths $\sigma_m^{p,\infty}(\mu_p)$.
In the case $m=0$ a stronger result on concentration of $\mu_p$ was obtained already in \cite[Theorem 3 and Remark 2]{SZ}.
It would be certainly of interest to obtain a similar statement also for other values of $m>0$,
but this would go beyond the scope of this paper and we leave this direction open for further study.
\item[(ii)] Theorem \ref{thm6} may be interpreted in the sense of the discussion after formula \eqref{eq:cal2}.
Namely, the average coordinate of $x\in\D$ is $n^{-1/p}$. Theorem \ref{thm6} shows, that the average value
of the largest coordinate is only slightly larger (namely $c[\ln(en)]^{1/p}$ times larger). In this sense, the average point
of $\D$ is only slightly modified (and properly normalized) white noise.
\item[(iii)]
Using the interpolation formula \eqref{eq:6}, 
one may immediately extend this result to all $0<p\le q<\infty$. But we shall see later on,
that in the case $q<\infty$, one may prove slightly better estimates.
\item[(iv)]
The behavior of $\sigma_{m}^{p,\infty}(\mu_p)$ was studied in detail in \cite[Example 10]{GLSW} for $p=2$.
It was shown that if $x_i$ are independent $N(0,1)$ Gaussian random variables and $m\le n/2+1$, then
\begin{align*}
c\sqrt{\ln \frac{2n}{m}}\le \E\, x_m^*\le C\sqrt{\ln \frac{2n}{m}},
\end{align*}
where $c$ and $C$ are absolute positive constants. Furthermore, if $m\ge n/2+1$, then
\begin{align*}
\sqrt\frac{\pi}{2}\ \frac{n-m+1}{n+1}\le \E\, x_m^*\le \sqrt{2\pi}\ \frac{n-m+1}{n}.
\end{align*}
\item[(v)] The method used in the proof of the second part of Theorem \ref{thm6}
may be found for example in \cite{G}.
\end{itemize}
\end{rem}

\subsection{The case $q<\infty$}
We discuss briefly also the case when $q<\infty$. It turns out, that in this case
the logarithmic term disappears. We do not go much into details and restrict ourselves to the case $m=0.$

\begin{prop}\label{prop12}
Let $n\ge 2$ and $0<p\le q<\infty$. 
Then
\begin{enumerate}
\item[(i)] $C_{p,q}^1n^{1/q}\le {\mathbb E}\, \|x\|_q \le C_{p,q}^2 n^{1/q}$,
\item[(ii)] 
$$
C_{p,q}^1\cdot \frac{{\mathbb E\,}\|x\|_q}{n^{1/p}}\le \sigma_0^{p,q}(\mu_p)=\int_{\D}\|x\|_q d\mu_p(x)
\le C_{p,q}^2 \cdot \frac{{\mathbb E\,}\|x\|_q}{n^{1/p}}
$$
and
\item[(iii)] $C_{p,q}^1 n^{1/q-1/p}\le \sigma_0^{p,q}(\mu_p)\le C_{p,q}^2n^{1/q-1/p}$,
\end{enumerate}
where in all these estimates $C_p^1$ and $C_p^2$ are positive real numbers depending only on $p$.

\end{prop}
\begin{proof} (i) The following two inequalities may be easily proved by H\"older's and Minkowski
inequality.
\begin{align*}
\biggl(\sum_{j=1}^n ({\mathbb E}x_j)^q\biggr)^{1/q}&\le
{\mathbb E}\bigl(\sum_{j=1}^n x_j^q\bigr)^{1/q}\le \bigl(\sum_{j=1}^n{\mathbb E}x_j^q\bigr)^{1/q},\qquad q\ge 1,\\
\bigl(\sum_{j=1}^n{\mathbb E}x_j^q\bigr)^{1/q}&\le
{\mathbb E}\bigl(\sum_{j=1}^n x_j^q\bigr)^{1/q}\le
\biggl(\sum_{j=1}^n ({\mathbb E}x_j)^q\biggr)^{1/q},\qquad q\le 1.
\end{align*}
This gives for $q\ge 1$
$$
{\mathbb E}\|x\|_q \le n^{1/q} ({\mathbb E}x_j^q)^{1/q} \quad\text{and}\quad
{\mathbb E}\|x\|_q \ge n^{1/q} {\mathbb E}x_j
$$
and for $q\le 1$
$$
{\mathbb E}\|x\|_q \le n^{1/q} {\mathbb E}x_j \quad\text{and}\quad
{\mathbb E}\|x\|_q \ge n^{1/q} ({\mathbb E}x^q_j)^{1/q}.
$$
Let us note, that the value of ${\mathbb E}x_j$ and $({\mathbb E}x^q_j)^{1/q}$ does not depend on $n$, only on $p$ and $q$.

(ii) The proof of the second part resembles very much the proof of Lemma \ref{lem1} and is left to the reader.

(iii) The last point follows immediately from (i) and (ii).
\end{proof}
\begin{rem}
A similar statement to Proposition \ref{prop12} is included in \cite[Lemma 2, point 4]{SZ}.
\end{rem}

\section{Normalized surface measure}

In this section we study the average best $m$-term widths for another classical measure on $\D$,
namely the normalized Hausdorff measure, cf. Definition \ref{dfn3}. Intuitively, this measure
gives more weight to those areas, where one or more components of $x\in\D$ are close to zero.
It turns out, that this is really the case - with the mathematical formulation given in Lemma \ref{lem7} below.
This relation is then used together with Lemma \ref{lem8} in Theorem \ref{thm9} to provide estimates of 
$\sigma_{0}^{p,\infty}(\varrho_p)$ from above.

\begin{dfn}\label{dfn3}
Let $n\ge 2$ be a natural number. We denote by
$$
\varrho_p({\mathcal A})=
\frac{{\mathcal H}({\mathcal A})}{{\mathcal H}(\D)},\qquad {\mathcal A}\subset \D
$$
the normalized $n-1$ dimensional Hausdorff measure on $\D$.
\end{dfn}
Let us mention, that for $p\in\{1,2,\infty\}$ the measure $\varrho_p$ coincides with $\mu_p.$
The following lemma provides a relationship between the normalized surface measure $\varrho_p$
and the cone measure $\mu_p$. For $p\ge 1$, it was given by \cite{NR}. We follow closely their approach
and it turns out, that it may be generalized also to the non-convex case of $0<p<1.$

\begin{lem}\label{lem7}
Let $0<p<\infty$ and $n\ge 2.$ Then $\varrho_p$ is an absolutely continuous measure with respect
to $\mu_p$ and for $\mu_p$ almost every $x\in\D$ it holds
$$
\frac{d\varrho_p} {d\mu_p}(x)=\frac{n\lambda([0,1]\cdot\D)}{{\mathcal H}(\D)}
\Bigl\|\nabla (\|\cdot\|_p)(x)\Bigr\|_2 = c^{-1}_{p,n} \biggl(\sum_{i=1}^{n}x_i^{2p-2}\biggr)^{1/2},
$$
where
$$
c_{p,n}=\int_{\D}\biggl(\sum_{i=1}^n x_i^{2p-2}\biggr)^{1/2}d \mu_p(x)
$$
is the normalizing constant.
\end{lem}
\begin{proof}
The proof imitates the proof of \cite[Lemma 1 and Lemma 2]{NR}, where the statement was proven for $1\le p<\infty$.
Hence, we may assume, that $0<p<1.$ First, we introduce some notation.

We fix $x=(x_1,\dots,x_n)\in \D$, such that
\begin{itemize}
\item the mapping $y\to \|y\|_p$ is differentiable at $x$,
\item $x$ is a density point of $\mathcal H$, i.e.
\begin{equation}\label{eq:pr7:0}
\lim_{\varepsilon\to 0+}\frac{{\mathcal H}(B(x,\varepsilon)\cap\D)}{\varepsilon^{n-1}V_{n-1}}=1,
\end{equation}
where $V_{n-1}$ denotes the Lebesgue volume of the $n-1$ dimensional Euclidean unit ball.
\item $x_i>0$ for all $i=1,\dots,n$.
\end{itemize}
Obviously, $\varrho_p$-almost every $x\in \D$ satisfies all the three properties 
(we refer for example to \cite[Theorem 16.2]{Mat} for the second one).

Furthermore, we put $z:=\nabla(\|\cdot\|_p)(x)$. This means, that
\begin{equation}\label{eq:pr7:1}
\|x+y\|_p=1+\langle z,y \rangle+r(y),
\end{equation}
where
$$
\theta(\delta):=\sup\left\{\frac{|r(y)|}{\|y\|_2}:0<\|y\|_2\le \delta\right\},\quad \delta>0
$$
tends to zero if $\delta$ tends to zero.
Using \eqref{eq:pr7:1} for $y=\delta x$, one observes, that $\langle z,x\rangle=1$.
We denote by $H=x+z^{\perp}$ the tangent hyperplane to $\D$ at $x$.
Let us note, that for $0<p<1$ the set $\R_+^n\setminus [0,1)\cdot \D=[1,\infty)\cdot\D$ is convex.
Next, we show, that $\langle z,y\rangle\ge 1$ for every $y\in [1,\infty)\cdot \D$. Indeed,
\begin{align*}
1&\le \|x+\lambda(y-x)\|_p=1+\langle z,\lambda(y-x)\rangle +r(\lambda(y-x))\\
&=1-\lambda+\lambda\langle z,y\rangle+r(\lambda(y-x))
\end{align*}
Dividing by $\lambda>0$ and letting $\lambda\to 0$ gives the statement.

The proof of the lemma is based on the following two inclusions, namely
\begin{equation}\label{eq:pr7:2}
[0,1] \cdot \Bigl(B(x,\varepsilon(1-\theta(\varepsilon)))\cap H\Bigr)
\subset [0,1]\cdot \Bigl(B(x,\varepsilon)\cap\D\Bigr)
\end{equation}
and
\begin{equation}\label{eq:pr7:2'}
[0,1]\cdot \Bigl(B(x,\varepsilon)\cap\D\Bigr)\subset [0,1+\varepsilon\theta(\varepsilon)]\cdot \Bigl( B(x,\varepsilon(1+\theta(\varepsilon)\|x\|_2))\cap H\Bigr),
\end{equation}
which hold for all $\varepsilon>0$ small enough.

First, we prove \eqref{eq:pr7:2}. To given $0\le s\le 1$ and
$v\in B(x,\varepsilon(1-\theta(\varepsilon))\cap H$ we need to find $0\le t\le 1$ and 
$w\in B(x,\varepsilon)\cap\D$, such that $sv=tw.$
To do this, we set 
$$
w:=\frac{v}{\|v\|_p}\in\D\quad\text{and}\quad t:=s\|v\|_p.
$$
We need to show, that $t\le 1$ and $\|x-w\|_2\le \varepsilon$.

We choose $0<\varepsilon\le \min_i x_i$. Then
$$
x_i\le |x_i-v_i|+v_i\le \|x-v\|_2+v_i\le \varepsilon+v_i
$$
for every $i=1,\dots,n$, which implies, that $v_i\ge 0$ and $v\in\R_+^n$. 
From $v\in H$ and $v\in\R_+^n$ we deduce, that $\|v\|_p\le 1$. Hence $t=s\|v\|_p\le \|v\|_p\le 1$.

Next, we write
\begin{align*}
\|x-w\|_2&=\Bigl\|x-\frac{v}{\|v\|_p} \Bigr\|_2\le \|x-v\|_2+\Bigl\|v-\frac{v}{\|v\|_p}\Bigr\|_2\\
&\le \varepsilon(1-\theta(\varepsilon))+\|v\|_2\cdot\frac{1-\|v\|_p}{\|v\|_p}\le \varepsilon(1-\theta(\varepsilon))+1-\|v\|_p\\
&=\varepsilon(1-\theta(\varepsilon))+1-\left\{1+\langle v-x,z\rangle+r(v-x)\right\}\\
&=\varepsilon(1-\theta(\varepsilon))+r(v-x)\le \varepsilon.
\end{align*}

Next, we prove \eqref{eq:pr7:2'}. 
We need to find to given
$0\le t\le 1$ and $w\in B(x,\varepsilon)\cap\D$ some $0\le s\le 1 +\varepsilon\theta(\varepsilon)$
and $v\in B(x,\varepsilon(1+\theta(\varepsilon)\|x\|_2))\cap H$, such that $tw=sv.$
We put
$$
s:=t\langle w,z\rangle\quad\text{and}\quad v:=\frac{w}{\langle w,z\rangle}.
$$
Let us recall, that we have shown above, that $w\in\D$ implies that $\langle w,z\rangle \ge 1.$

Of course, $tw=sv$ and $v\in H$ (as $\langle v,z\rangle =1$). Hence, it remains to show, that
$s\le 1+\varepsilon \theta(\varepsilon)$ and $\|v-x\|_2\le \varepsilon(1+\theta(\varepsilon)\|x\|_2)$.

The application of \eqref{eq:pr7:1} gives
$$
1=\|w\|_p=\|x+(w-x)\|_p=1+\langle w-x,z\rangle + r(w-x),
$$
which again forces $\langle w,z\rangle\le 1+\varepsilon \theta(\varepsilon).$ 
Then $s=t\langle w,z\rangle \le \langle w,z\rangle \le 1+\varepsilon\theta(\varepsilon).$

Finally, we write
\begin{align*}
\|v-x\|_2&=\Bigl|\Bigl| \frac{w}{\langle w,z\rangle} - x\Bigr\|_2
\le \Bigl\|\frac{w}{\langle w,z\rangle}-\frac{x}{\langle w,z\rangle} \Bigr\|_2+
\Bigl\|\frac{x}{\langle w,z\rangle} -x\Bigr\|_2\\
&\le \frac{\|w-x\|_2}{\langle w,z\rangle}+\|x\|_2\frac{\langle w,z\rangle-1}{\langle w,z\rangle}
\le \varepsilon+\varepsilon\theta(\varepsilon)\|x\|_2.
\end{align*}
Equipped with \eqref{eq:pr7:2} and \eqref{eq:pr7:2'}, we may finish the proof of the lemma. We write
\begin{align}\notag
\lim_{\varepsilon\to 0}\frac{\varrho_p(B(x,\varepsilon)\cap\D)}{\mu_p(B(x,\varepsilon)\cap\D)}&=
\lim_{\varepsilon\to 0}\frac{{\mathcal H}(B(x,\varepsilon)\cap\D)}{{\mathcal H}(\D)}
\cdot\frac{\varepsilon^{n-1}V_{n-1}}{\varepsilon^{n-1}V_{n-1}}\cdot
\frac{\lambda([0,1]\cdot\D)}{\lambda([0,1]\cdot[B(x,\varepsilon)\cap\D])}\\
&\label{eq:pr7:3}=\frac{\lambda([0,1]\cdot\D)}{{\mathcal H}(\D)}\cdot
\lim_{\varepsilon\to 0}\frac{\varepsilon^{n-1}V_{n-1}}{\lambda([0,1]\cdot [B(x,\varepsilon)\cap\D])},
\end{align}
where we have used \eqref{eq:pr7:0}. As the perpendicular distance between zero and $H$ is equal to
$1/\|z\|_2$, we observe, that
$$
{\rm vol}(B(x,a)\cap H)=\frac{a^{n-1}V_{n-1}}{n\|z\|_2}
$$
holds for every $a>0.$ Using this, we get from \eqref{eq:pr7:2} and \eqref{eq:pr7:2'}
\begin{align*}
\lambda&\left( [0,1]\cdot \Bigl(B(x,\varepsilon(1-\theta(\varepsilon)))\cap H\Bigr)\right)=
\frac{[\varepsilon(1-\theta(\varepsilon))]^{n-1}V_{n-1}}{n\|z\|_2}\\
&\qquad \le \lambda \left([0,1]\cdot \Bigl(B(x,\varepsilon)\cap\D\Bigr)\right)\\
&\qquad\le \lambda \left([0,1+\varepsilon\theta(\varepsilon)]\cdot\Bigl( B(x,\varepsilon(1+\theta(\varepsilon)\|x\|_2))\cap H\Bigr)\right)\\
&\qquad=[1+\varepsilon\theta(\varepsilon)]^n\cdot \frac{[\varepsilon(1+\theta(\varepsilon)\|x\|_2)]^{n-1}V_{n-1}}{n\|z\|_2}.
\end{align*}
Combining these estimates with \eqref{eq:pr7:3} gives the result.
\end{proof}

Following lemma is analogous to Lemma \ref{lem1} and reduces the calculation of $\sigma_0^{p,\infty}(\varrho_p)$
to inequalities for the estimated values of functions of the random variables $x_1,\dots,x_n$.

\begin{lem}\label{lem8} 
Let $0<p<\infty$. There exists two positive real numbers $C_p^1$ and $C_p^2$, such that  
\begin{align}\label{eq:8}
C_p^1\cdot& \frac{\displaystyle{\mathbb E\,} x_1^*\Bigl(\sum_{i=1}^n x_i^{2p-2}\Bigr)^{1/2}}
{\displaystyle{\mathbb E}\Bigl(\sum_{i=1}^n x_i^{2p-2}\Bigr)^{1/2}}\cdot n^{-1/p}\le 
\sigma^{p,\infty}_{0}(\varrho_p)
=\int_{\D}x_1^* d\varrho_p\\&=
\frac{\displaystyle \int_{\D}x^*_1\Bigl(\sum_{i=1}^n x_i^{2p-2}\Bigr)^{1/2}d\mu_p(x)}
{\displaystyle \int_{\D}\Bigl(\sum_{i=1}^n x_i^{2p-2}\Bigr)^{1/2}d\mu_p(x)}
\notag \le C_p^2 \frac{\displaystyle{\mathbb E\,} x_1^*\Bigl(\sum_{i=1}^n x_i^{2p-2}\Bigr)^{1/2}}
{\displaystyle{\mathbb E}\Bigl(\sum_{i=1}^n x_i^{2p-2}\Bigr)^{1/2}}\cdot n^{-1/p}
\end{align}
\end{lem}
for all $n\ge 2.$
\begin{proof}
Only the inequalities need a proof. It resembles the proof of Lemma \ref{lem1}
and is again based on the polar decomposition formula \eqref{eq:4}.

We plug the functions
$$
f_1(x)=x_1^*\Bigl(\sum_{i=1}^n x_i^{2p-2}\Bigr)^{1/2}e^{-x_1^p-\dots-x_n^p}
\quad\text{and}\quad
f_2(x)=\Bigl(\sum_{i=1}^n x_i^{2p-2}\Bigr)^{1/2}e^{-x_1^p-\dots-x_n^p}
$$
into \eqref{eq:4} and obtain
\begin{align*}
\sigma^{p,\infty}_{0}(\varrho_p)&=
\frac{\displaystyle \int_{\R_+^n} f_1(x)dx \cdot \int_0^\infty r^{n+p-2}e^{-r^p}dr}
{\displaystyle\displaystyle \int_{\R_+^n} f_2(x)dx \cdot \int_0^\infty r^{n+p-1}e^{-r^p}dr}\\
&=\frac{\displaystyle{\mathbb E\,} x_1^*\Bigl(\sum_{i=1}^n x_i^{2p-2}\Bigr)^{1/2}}
{\displaystyle{\mathbb E}\Bigl(\sum_{i=1}^n x_i^{2p-2}\Bigr)^{1/2}}\cdot 
\frac{\Gamma(n/p+1-1/p)}{\Gamma(n/p+1)}.
\end{align*}
By Stirling's formula, the last expression is equivalent to $n^{-1/p}$ with constants of equivalence depending only on $p$.
\end{proof}

\begin{thm}\label{thm9}
Let $0<p< \infty$. Then there is a positive real number $C_p$, such that
\begin{equation}\label{eq:9}
\sigma_{0}^{p,\infty}(\varrho_p) \le C_p \left[\frac{\log(n+1)}{n}\right]^{1/p}
\end{equation}
for all $n\ge 2.$
\end{thm}
\begin{proof}
We define a probability measure $\alpha_{p,n}$ on $\R^+_n$ by the density
\begin{equation*}
\tilde c_{p,n}^{-1}\cdot\left(\sum_{i=1}^nx_i^{2p-2}\right)^{1/2}e^{-x_1^p-\dots-x_n^p},\quad
\tilde c_{p,n}:=\int_{\R_+^n}\left(\sum_{i=1}^nx_i^{2p-2}\right)^{1/2}e^{-x_1^p-\dots-x_n^p}dx
\end{equation*}
with respect to the Lebesgue measure. Let us note, that due to the inequality
$$
\left(\sum_{i=1}^nx_i^{2p-2}\right)^{1/2}\le \sum_{i=1}^n x_i^{p-1}
$$
the integral in the definition of $\tilde c_{p,n}$ really converges and $\alpha_{p,n}$ is well defined.

According to Lemma \ref{lem8}, we need to estimate
$$
\int_{\R_+^n}x_1^* d\alpha_{p,n}(x).
$$
We calculate for $\delta>1$, which is to be chosen later on,
\begin{align*}
\int_{\R_+^n}x_1^* d\alpha_{p,n}(x)&=\int_0^\infty \alpha_{p,n}(x_1^*>t)dt
\le \delta+\int_{\delta}^\infty \alpha_{p,n}(x_1^*>t)dt\\
&\le \delta+n\int_{\delta}^\infty\alpha_{p,n}(x_1>t)dt.
\end{align*}
We write $x'=(x_2,\dots,x_n)\in\R_+^{n-1}$. Then
\begin{align*}
\alpha_{p,n}(x_1>t)&=\tilde c_{p,n}^{-1}\int_{t}^\infty e^{-x_1^p}\int_{\R_+^{n-1}}\left(\sum_{i=1}^nx_i^{2p-2}\right)^{1/2}e^{-x_2^p-\dots-x_n^p}dx' dx_1\\
&\le \tilde c_{p,n}^{-1}\int_{t}^\infty e^{-x_1^p}\int_{\R_+^{n-1}}\left[x_1^{p-1}+\left(\sum_{i=2}^nx_i^{2p-2}\right)^{1/2}\right]e^{-x_2^p-\dots-x_n^p}dx' dx_1\\
&=\tilde c_{p,n}^{-1}\int_{t}^\infty e^{-x_1^p}x_1^{p-1}dx_1\cdot \int_{\R_+^{n-1}}e^{-x_2^p-\dots-x_n^p}dx'\\
&\qquad+\tilde c_{p,n}^{-1}\int_{t}^\infty e^{-x_1^p}dx_1\cdot \int_{\R_+^{n-1}}\left(\sum_{i=2}^nx_i^{2p-2}\right)^{1/2}e^{-x_2^p-\dots-x_n^p}dx'\\
&:=I_1+I_2.
\end{align*}
The inequality
\begin{align}
\notag c_p^n \tilde c_{p,n}&=c_p^n \int_{\R_+^n}\left(\sum_{i=1}^nx_i^{2p-2}\right)^{1/2}e^{-x_1^p-\dots-x_n^p}dx\\
\label{eq:in1}&\ge c_p^n \int_{\R_+^n}\left(\sum_{i=2}^nx_i^{2p-2}\right)^{1/2}e^{-x_1^p-\dots-x_n^p}dx\\
\notag&=c_p^n \int_0^\infty e^{-x_1^p}dx_1\int_{\R_+^{n-1}}\left(\sum_{i=2}^nx_i^{2p-2}\right)^{1/2}e^{-x_2^p-\dots-x_n^p}dx'
=c_p^{n-1}\tilde c_{p,n-1}
\end{align}
shows, that
$$
I_1=\frac{c_p\int_{t}^\infty x_1^{p-1}e^{-x_1^p}dx_1}{c_p^n \tilde c_{p,n}}\le 
\frac{c_p\int_{t}^\infty x_1^{p-1}e^{-x_1^p}dx_1}{c_p \tilde c_{p,1}}
=\tilde c_{p,1}^{-1}\cdot\frac{e^{-t^p}}{p}.
$$
Using \eqref{eq:in1} again, we get also
$$
I_2=\tilde c_{p,n}^{-1}\cdot \tilde c_{p,n-1}\int_t^\infty e^{-x_1^p}dx_1
\le c_p\int_t^\infty e^{-x_1^p}dx_1=\frac{c_p}{p}\cdot\int_{t^p}^\infty s^{1/p-1}e^{-s}ds.
$$
If $p\ge 1$, we get
\begin{equation}\label{eq:kvak1}
I_1+I_2\le C_p e^{-t^p},\quad t>1
\end{equation}
and
$$
\int_{\R_+^n}x_1^* d\alpha_{p,n}(x)\le \delta + C_pn\int_{\delta}^\infty e^{-t^p}dt
\le \delta +C_p'ne^{-\delta^p}.
$$
By choosing $\delta=C_p \log(n+1)^{1/p}$, we get the result.

If $p<1$, we use the second estimate of Lemma \ref{lem3} and replace \eqref{eq:kvak1} with
$$
I_1+I_2\le C_pt^{1-p}e^{-t^p},\quad t>t_0
$$
for $t_0>1$ large enough and the result again follows by the choice of $\delta.$

\end{proof}

\begin{rem}
\begin{itemize}
\item[(i)] Theorem \ref{thm9} shows, that the average size of the largest coordinate of $x\in\D$
taken with respect to the normalized Hausdorff measure is again only slightly larger than $n^{-1/p}$.
Hence, also in this case, the typical element of $\D$ seems to be far from being sparse
and resembles rather properly normalized white noise in the sense described in Introduction.
\item[(ii)] Using interpolation inequality \eqref{eq:6}, one may again obtain a similar estimate also for $0<p\le q<\infty$, namely
$$
\sigma_{0}^{p,q}(\varrho_p) \le C_{p,q} \left[\frac{\log(n+1)}{n}\right]^{1/p-1/q}.
$$
It would be probably possible to avoid the logarithmic terms and provide improved estimates also for $m>0$, but we shall
not go into this direction. Our main aim of this section was to show, that normalized Hausdorff measure does not
prefer sparse (or nearly sparse) vectors, and this was clearly demonstrated by Theorem \ref{thm9}.
\end{itemize}
\end{rem}

\section{Tensor product measures}

As discussed already in the Introduction and proved in Theorem \ref{thm6} and Theorem \ref{thm9},
the average vectors of $\D$ with respect to the cone measure $\mu_p$ and with respect to
surface measure $\varrho_p$ behave ``badly'' meaning that (roughly speaking) many of their coordinates
are approximately of the same size. As promised before, we shall now introduce a new class of measures,
for which the random vector behaves in a completely different way. These measures are defined through
their density with respect to the cone measure $\mu_p$. This density has a strong singularity near the points
with vanishing coordinates.

\begin{dfn}\label{def14}
Let $0<p<\infty$, $\beta>-1$ and $n\ge 2$. Then we define the probability measure $\theta_{p,\beta}$
on $\D$ by
\begin{equation}\label{eq:11}
\frac{d \theta_{p,\beta}}{d\mu_p}(x)=c^{-1}_{p,\beta}\cdot \prod_{i=1}^n x_i^{\beta}, \quad x\in\D,
\end{equation}
where
\begin{equation}\label{eq:12}
c_{p,\beta}=\int_{\D}\prod_{i=1}^{n} x_i^{\beta} d\mu_p(x).
\end{equation}
\end{dfn}
\begin{rem}\label{rem4}
\begin{itemize}
\item[(i)] If $0>\beta>-1$, then \eqref{eq:11} defines the density of $\theta_{p,\beta}$
with respect to $\mu_p$ only for points, where $x_i\not=0$ 
for all $i=1,\dots, n$. That means, that this density is defined 
$\mu_p$-almost everywhere. The definition is then complemented by the statement,
that $\theta_{p,\beta}$ is absolutely continuous with respect to $\mu_p$.

\item[(ii)] We shall see later on, that the condition $\beta>-1$ ensures, that 
\eqref{eq:12} is finite.

\item[(iii)] It was observed already in \cite{BCN}, that the measures $\theta_{p,\beta}$ allow a formula similar to \eqref{eq:gener1}.
We plug the function $f(x)=\chi_{[0,\infty)\cdot {\mathcal A}}\prod_{i=1}^n x_i^\beta e^{-\|x\|_p^p}$ into \eqref{eq:4}, where ${\mathcal A}$ 
is any $\mu_p$-measurable subset of $\D$, and obtain
$$
\int_{[0,\infty)\cdot {\mathcal A}}\prod_{i=1}^nx_i^\beta e^{-\|x\|_p^p}d\lambda(x)=\lambda([0,1]\cdot \D)\cdot n\cdot\int_0^\infty r^{n-1+n\beta}e^{-r^p}dr
\cdot\int_{\mathcal A}\prod_{i=1}^nx_i^\beta d\mu_p(x).
$$
We use a similar formula also for ${\mathcal A}=\D$, which leads to
$$
\int_{{\mathcal A}}1d\,\theta_{p,\beta}=\frac{\displaystyle\int_{{\mathcal A}}\prod_{i=1}^nx_i^\beta d\mu_p(x)}{\displaystyle\int_{\D}\prod_{i=1}^nx_i^\beta d\mu_p(x)}
=\frac{\displaystyle \int_{[0,\infty)\cdot {\mathcal A}}\prod_{i=1}^nx_i^\beta e^{-\|x\|_p^p}dx}
{\displaystyle \int_{\R_+^n}\prod_{i=1}^n x_i^\beta e^{-\|x\|_p^p}dx}.
$$
Let $\omega'=(\omega_1',\dots,\omega_n')$ be a vector with independent identically distributed components
with respect to the density $c_{p,\beta}t^\beta e^{-t^p},t>0,$ where $c^{-1}_{p,\beta}=\int_0^\infty t^\beta e^{-t^p}dt$
is a normalizing constant. Up to a simple substitution, this is the well known \emph{gamma distribution}.
We observe that the distribution of random points with respect to $\theta_{p,\beta}$ equals to the distribution
of $\ell_p^n$ normalized vectors $\omega'$, i.e.
\begin{equation}\label{eq:generate}
\theta_{p,\beta}({\mathcal A})=
{\mathbb P}\Biggl(\frac{(\omega'_1,\dots,\omega'_n)}{\bigl(\sum_{j=1}^n{(\omega'_j)^p}\bigr)^{1/p}}\in 
{\mathcal A}\Biggr), \qquad {\mathcal A}\subset\D.
\end{equation}
\item[(iv)] Of course, the same procedure might be considered also for other distributions. We leave this to future work.
We also refer to the discussion on the recent work of Gribonval, Cevher, and Davies \cite{GCD} in the Introduction.
\end{itemize}
\end{rem}

\begin{lem}\label{lem15}
Let $0<p<\infty$, $\beta>-1$ and $n\ge 2$.
\begin{itemize}
\item[(i)] Let $1\le m \le n$. Then
$$
\sigma_{m-1}^{p,\infty}(\theta_{p,\beta})=
\int_{\D} x_m^* d\theta_{p,\beta}
=\frac{\displaystyle\E\, x_m^* \prod_{i=1}^n x_i^\beta}{\displaystyle\E\prod_{i=1}^n x_i^\beta}\cdot
\frac{\Gamma(n(\beta+1)/p)}{\Gamma(n(\beta+1)/p+1/p)}.
$$
\item[(ii)]
$$
\E\prod_{i=1}^n x_i^\beta=\left[\frac{c_p}{p}\cdot\Gamma((\beta+1)/p)\right]^n.
$$
\end{itemize}
\end{lem}
\begin{proof}
The proof of the first part follows again by \eqref{eq:4}, this time used for the functions
$$
f_1(x)=x_m^*\Bigl(\prod_{i=1}^n x_i^\beta\Bigr) e^{-x_1^p-\dots-x_n^p}
\quad\text{and}\quad
f_2(x)=\Bigl(\prod_{i=1}^n x_i^\beta\Bigr) e^{-x_1^p-\dots-x_n^p}.
$$
The proof of the second part is straightforward.
\end{proof}
It follows directly from \eqref{eq:Gamma}, that $\Gamma(s)$ tends to infinity, when $s$ tends to zero.
The following lemma quantifies this phenomenon. Although the statement seems to be well known, we were
not able to find a reference and we therefore provide at least a sketch of the proof.
\begin{lem}\label{lem17}
Let $C\simeq 0.577\dots$ denote the Euler constant. Then
$$
\lim_{n\to\infty}\left(\frac{\Gamma(1/n)}{n}\right)^n=e^{-C}.
$$
\end{lem}
\begin{proof}
It is enough to show, that
$$
\lim_{n\to\infty}n\cdot \log(\Gamma(1+1/n))=-C,
$$
which (by using the l'Hospital rule) follows from
$$
\lim_{n\to\infty}\frac{\int_0^\infty s^{1/n}e^{-s}\log s\,ds}{\int_0^\infty s^{1/n}e^{-s}ds}=-C.
$$
But the numerator of this fraction is equal to $\Gamma'(1+1/n)$ and its denominator to $\Gamma(1+1/n)$. The whole fraction
is therefore equal to $\Psi(1+1/n)$ and $\Psi(1+1/n)\to \Psi(1)=-C$ as $n$ tends to infinity, cf. \cite[Section 6.3.2, p. 258]{AS}.
\end{proof}
Next theorem shows, that if $\beta=p/n-1$, then the measure $\theta_{p,\beta}$ promotes sparsity and one may even consider limiting behavior
of $n$ growing to infinity.
\begin{thm} \label{thm17}
Let $0<p<\infty$ and let $n\ge 2$ and $1\le m\le n$ be integers.
Then
\begin{equation}\label{eq:F1}
\sigma_{m-1}^{p,\infty}(\theta_{p,p/n-1}) \ge C^1_p \cdot \frac{\Gamma(n+1)}{\Gamma(n-m+1)}\cdot\frac{\Gamma(n/p+n-m+1)}{\Gamma(n/p+n+1)},
\end{equation}
and
\begin{equation}\label{eq:F2}
\sigma_{m-1}^{p,\infty}(\theta_{p,p/n-1}) \le C_p^2 \cdot \frac{\Gamma(n+1)}{\Gamma(n-m+1)}
\left\{\frac{\Gamma(n/p+n-m+1)}{\Gamma(n/p+n+1)}+\frac{1}{m!}\cdot\left(\frac{e^{-1}}{\Gamma(1/n)}\right)^m\right\}
\end{equation}
where $C^1_p$ and $C_p^2$ are positive real numbers depending only on $p$.

Furthermore, for every fixed $m\in\N$,
\begin{align}\label{eq:final1}
\frac{C_p^1}{\left(\frac{1}{p}+1\right)^m}&\le \liminf_{n\to\infty} \sigma_{m-1}^{p,\infty}(\theta_{p,p/n-1})
\le \limsup_{n\to\infty} \sigma_{m-1}^{p,\infty}(\theta_{p,p/n-1}) \le
\frac{C_p^2}{\left(\frac{1}{p}+1\right)^m},
\end{align}
where $C^1_p$ and $C_p^2$ are positive real numbers depending only on $p$.
\end{thm}
\begin{proof}
First observe, that $n(\beta+1)/p=1$ for $\beta=p/n-1$ and therefore
$$
\frac{\Gamma(n(\beta+1)/p)}{\Gamma(n(\beta+1)/p+1/p)}=\frac{1}{\Gamma(1+1/p)}
$$
depends only on $p$. Due to Lemma \ref{lem15}, we have to estimate
\begin{equation}\label{eq:kvak2}
\E\, x_m^* \biggl(\prod_{i=1}^n x_i^{p/n-1}\biggr)=
c_p^n\int_{\R^d_+}x_m^*\prod_{i=1}^n x_i^{p/n-1} e^{-x_1^p-\dots-x_n^p}dx.
\end{equation}
Let $t=x_m^*$ and let us assume, that there is only one coordinate $j=1,\dots,n$, such that $x_j=t$.
Obviously, this assumption holds almost everywhere. Of course, we have $n$ possibilities for $j$.
Furthermore, $m-1$ from the remaining $n-1$ components of $x$ are bigger than $t$ and the remaining 
$n-m$ components are smaller.
This allows to rewrite \eqref{eq:kvak2} as
\begin{align*}
c_p^n\, n\binom{n-1}{m-1}&\int_0^\infty t^{p/n}e^{-t^p}\biggl(\int_0^t u^{p/n-1} e^{-u^p}du\biggr)^{n-m}\times\\
&\times \biggl(\int_t^\infty u^{p/n-1} e^{-u^p}du\biggr)^{m-1}dt\\
=\frac{c_p^n n}{p^n}\binom{n-1}{m-1}&\int_0^\infty \omega^{1/p+1/n-1}e^{-\omega}
\biggl(\int_0^\omega s^{1/n-1}e^{-s}ds\biggr)^{n-m}\times\\
&\times\biggl(\int_\omega^\infty s^{1/n-1}e^{-s}ds\biggr)^{m-1}d\omega.
\end{align*}
Let us denote 
\begin{equation}\label{eq:g:1}
\gamma=\Gamma(1/n)=\int_0^\infty s^{1/n-1}e^{-s}ds\quad\text{and}\quad
y(\omega)=\gamma^{-1}\cdot\int_0^\omega s^{1/n-1}e^{-s}ds.
\end{equation}
Then $y(\omega)$ is a non-decreasing function of $\omega$, $y(0)=0$ and $\lim_{\omega\to\infty}y(\omega)=1.$
We denote by $\omega(y)$ its inverse function, i.e. 
\begin{equation}\label{eq:g:2}
y=\gamma^{-1}\cdot\int_0^{\omega(y)} s^{1/n-1}e^{-s}ds,\quad 0\le y\le 1.
\end{equation}
Using this notation, we obtain
$$
\E\, x_m^* \biggl(\prod_{i=1}^n x_i^{p/n-1}\biggr)=
\frac{c_p^n\,\gamma^n}{p^n}\, n\binom{n-1}{m-1}\int_0^1 \omega(y)^{1/p} y^{n-m}(1-y)^{m-1}dy
$$
and
\begin{equation}\label{eq:kvak3}
\sigma_{m-1}^{p,\infty}(\theta_{p,p/n-1})=\frac{\Gamma(n+1)}{\Gamma(m)\Gamma(n-m+1)}
\int_0^1\omega(y)^{1/p}y^{n-m}(1-y)^{m-1}dy,
\end{equation}
where $\omega(y)$ is given by \eqref{eq:g:2}.

{\it Step 1. Estimate from below}

The estimate 
$$
\gamma y=\int_0^{\omega(y)} s^{1/n-1}e^{-s} ds\le \int_0^{\omega(y)} s^{1/n-1}ds=n\omega(y)^{1/n}
$$
implies together with Lemma \ref{lem17}
$$
\omega(y)\ge \left(\frac{\gamma y}{n}\right)^n \ge cy^n
$$
with $c$ independent of $n$. This gives finally
\begin{align*}
\sigma_{m-1}^{p,\infty}(\theta_{p,p/n-1})&\ge c^{1/p}\cdot\frac{\Gamma(n+1)}{\Gamma(m)\Gamma(n-m+1)}
\cdot \int_0^1 y^{n/p+n-m}(1-y)^{m-1}dy\\
&=c^{1/p}\cdot\frac{\Gamma(n+1)}{\Gamma(m)\Gamma(n-m+1)}\cdot B(n/p+n-m+1,m)\\
&=c^{1/p}\cdot\frac{\Gamma(n+1)}{\Gamma(n-m+1)}\cdot\frac{\Gamma(n/p+n-m+1)}{\Gamma(n/p+n+1)},
\end{align*}
where we used the Beta function \eqref{eq:Beta} and the proof of \eqref{eq:F1} is complete.

{\it Step 2. Estimate from above}\\
Let us first take $y$, such that $1-e^{-1}/\gamma\le y\le 1$. Then $-\ln(\gamma(1-y))\ge 1$ and
$$
\int_{-\ln(\gamma(1-y))}^\infty s^{1/n-1}e^{-s}ds\le 
\int_{-\ln(\gamma(1-y))}^\infty e^{-s}ds=\gamma(1-y).
$$
Hence, 
\begin{equation}\label{eq:omega:1}
\omega(y)\le -\ln(\gamma(1-y)),\quad 1-e^{-1}/\gamma\le y\le 1.
\end{equation}

Finally, we observe, that
$$
f:y\to \int_{Cy^n}^\infty s^{1/n-1}e^{-s}ds
$$
is a convex function on $\R_+$, $f(0)=\gamma$ and
\begin{align*}
f(1-e^{-1}/\gamma)&=\int_{C(1-e^{-1}/\gamma)^n}^\infty s^{1/n-1}e^{-s}ds\\
&\le \int_1^\infty s^{1/n-1}e^{-s}ds\le e^{-1},
\end{align*}
if we choose $C$ so large, that $C(1-e^{-1}/\gamma)^n\ge 1$ for all $n\in\N.$ This is indeed possible, while a byproduct of Lemma
\ref{lem17} is also a relation $\lim_{n\to\infty}\gamma/n=1.$ Using the convexity of $f$, we obtain
$$
f(y)\le \gamma(1-y),\quad 0\le y \le 1-e^{-1}/\gamma,
$$
which further leads to
\begin{equation}\label{eq:omega:2}
\omega(y)\le Cy^n,\quad 0\le y\le 1-e^{-1}/\gamma.
\end{equation}

We insert \eqref{eq:omega:1} and \eqref{eq:omega:2} into \eqref{eq:kvak3} and obtain
\begin{equation}\label{eq:omega:3}
\sigma_{m-1}^{p,\infty}(\theta_{p,p/n-1})\le \frac{\Gamma(n+1)}{\Gamma(m)\Gamma(n-m+1)}\left\{C^{1/p} I_1+I_2\right\},
\end{equation}
where
$$
I_1:=\int_0^{1-e^{-1}/\gamma} y^{n/p+n-m}(1-y)^{m-1}dy
$$
and
$$
I_2:=\int_{1-e^{-1}/\gamma}^1 |\ln(\gamma(1-y))|^{1/p}y^{n-m}(1-y)^{m-1}dy.
$$
The first integral may be estimated again using the Beta function, which gives
\begin{equation}\label{eq:omega:4}
I_1\le B(n/p+n-m+1, m).
\end{equation}
We denote by $k$ the uniquely defined integer, such that $1/p\le k<1/p+1$ holds, and estimate
\begin{align*}
I_2\le \int_{1-e^{-1}/\gamma}^1 |\ln(\gamma(1-y))|^{1/p}(1-y)^{m-1}dy
\le I_{k,m}:=\int_{0}^{e^{-1}/\gamma} |\ln(\gamma y)|^{k}y^{m-1}dy.
\end{align*}
Next, we use partial integration to estimate $I_{k,m}$. We obtain
$$
I_{k,m}=\frac{1}{m}\left(\frac{e^{-1}}{\gamma}\right)^m+\frac{k}{m}\cdot I_{k-1,m}.
$$
Together with $I_{0,m}=1/m\cdot (e^{-1}/\gamma)^m$, this leads finally to
$$
I_{k,m}\le \frac{(k+1)!}{m}\left(\frac{e^{-1}}{\gamma}\right)^m.
$$
This, together with \eqref{eq:omega:3} and \eqref{eq:omega:4} finishes the proof of \eqref{eq:F2}.

The proof of \eqref{eq:final1} then follows directly by Stirling's formula \eqref{eq:Stirl}.
\end{proof}

\begin{rem}
\begin{itemize}
\item[(i)] Let us take $m=0$. Then the formula \eqref{eq:final1} describes an essentially different behavior compared to the normalized cone and surface measure.
Namely, the expected value of the largest coordinate of $x\in \D$ with respect to $\theta_{p,p/n-1}$ does not decay to zero with
$n$ growing to infinity. We shall demonstrate this effect also numerically in next section.

\item[(ii)] If $m>0$, then \eqref{eq:final1} shows, that $\sigma_m^{p,\infty}(\theta_{p,p/n-1})$ decays exponentially fast with $m$,
as soon as $n$ is large enough. That means, that for $n$ large enough, the average vector of $\D$ exhibits a strong sparsity-like structure.
Namely, its $m$-th largest component decays exponentially with $m$.

\item[(iii)] We have chosen in \eqref{eq:11} a different $\beta$ for each $n$, namely $\beta_n=p/n-1>-1$. This was of course a crucial ingredient
in the proof of Theorem \ref{thm17}. It is not difficult to modify the analysis of the proof of Theorem \ref{thm17} to the situation, when
$\beta>-1$ is fixed for all $n\in\N$. In this case we obtain again, that (up to logarithmic factors) $\sigma_{0}^{p,\infty}(\theta_{p,\beta})$
is equivalent to $n^{-1/p}$ with constants of equivalence depending on $p>0$ and $\beta>-1.$ 

\item[(iv)] Last, but not least, we observe, that one may choose $p=1$ or even $p=2$ in Theorem \ref{thm17} and still obtains
the exponential decay of coordinates as described by \eqref{eq:final1}. 
It seems, that there is no significant connection between sparsity of an average vector of $x\in \D$
and the size of $p>0.$
\end{itemize}
\end{rem}
\section{Numerical experiments}
\subsection{Cone measure}
We would like to demonstrate the most significant effects of the theory also by numerical experiments. We start with the case of the cone measure.
The key role is played by \eqref{eq:gener1}. It may be interpreted in the following way. To generate a random point on $\D$
with respect to the normalized cone measure, it is enough to generate $\omega_1,\dots,\omega_n$ with respect
to the density $c_pe^{-t^p},t>0$ and then calculate
$$
\frac{(\omega_1,\dots,\omega_n)}{\bigl(\sum_{j=1}^n{\omega_j^p}\bigr)^{1/p}}\in\D.
$$
This method is very practical, as the running time of this algorithm depends only linearly on $n$.

Let us note, that the values of $\omega_i$ may be generated very easily. For example the package \emph{GNU Scientific Library}
\cite{GSL} implements a random number generator with respect to the gamma distribution using the method described
in the classical work of Knuth \cite{K}. Using this package, we generated $10^8$ random points $x\in \D$ for
$n=100$ and $p\in\{1/2,1,2\}$ to approximate numerically the value of $n^{1/p}\cdot \int_{\D}x_m^* d\mu_p(x)$. The result
may be found in the Figure \ref{fig:1}.

\subsection{Tensor measures}

As pointed out in Remark \ref{rem4}, point (iii), a random point on $\D$ with respect to $\theta_{p,\beta}$ may 
be generated in the following way. We generate $\omega'_1,\dots,\omega'_n$
with respect to the density $c_{p,\beta}t^\beta e^{-t^p},t>0$,
where $c^{-1}_{p,\beta}=\int_0^\infty t^\beta e^{-t^p}dt$ is a normalizing constant
and we consider the vector
$$
\frac{(\omega'_1,\dots,\omega'_n)}{\bigl(\sum_{j=1}^n(\omega'_j)^p\bigr)^{1/p}}\in\D.
$$
Also this may be easily done with the help of \cite{GSL}. We generated again $10^8$ random points $x\in\D$ with respect to 
$\theta_{p,p/n-1}$ for $n=100$ and $p\in\{1/2,1,2\}$. Then we used those points to numerically approximate
the expression $\log_{10}(\int_{\D}x_m^* d\theta_{p,p/n-1}).$

\begin{figure}[h]
  \subfloat[$n^{1/p}\cdot \int_{\D}x_m^* d\mu_p(x)$]{\includegraphics[width=.5\textwidth]{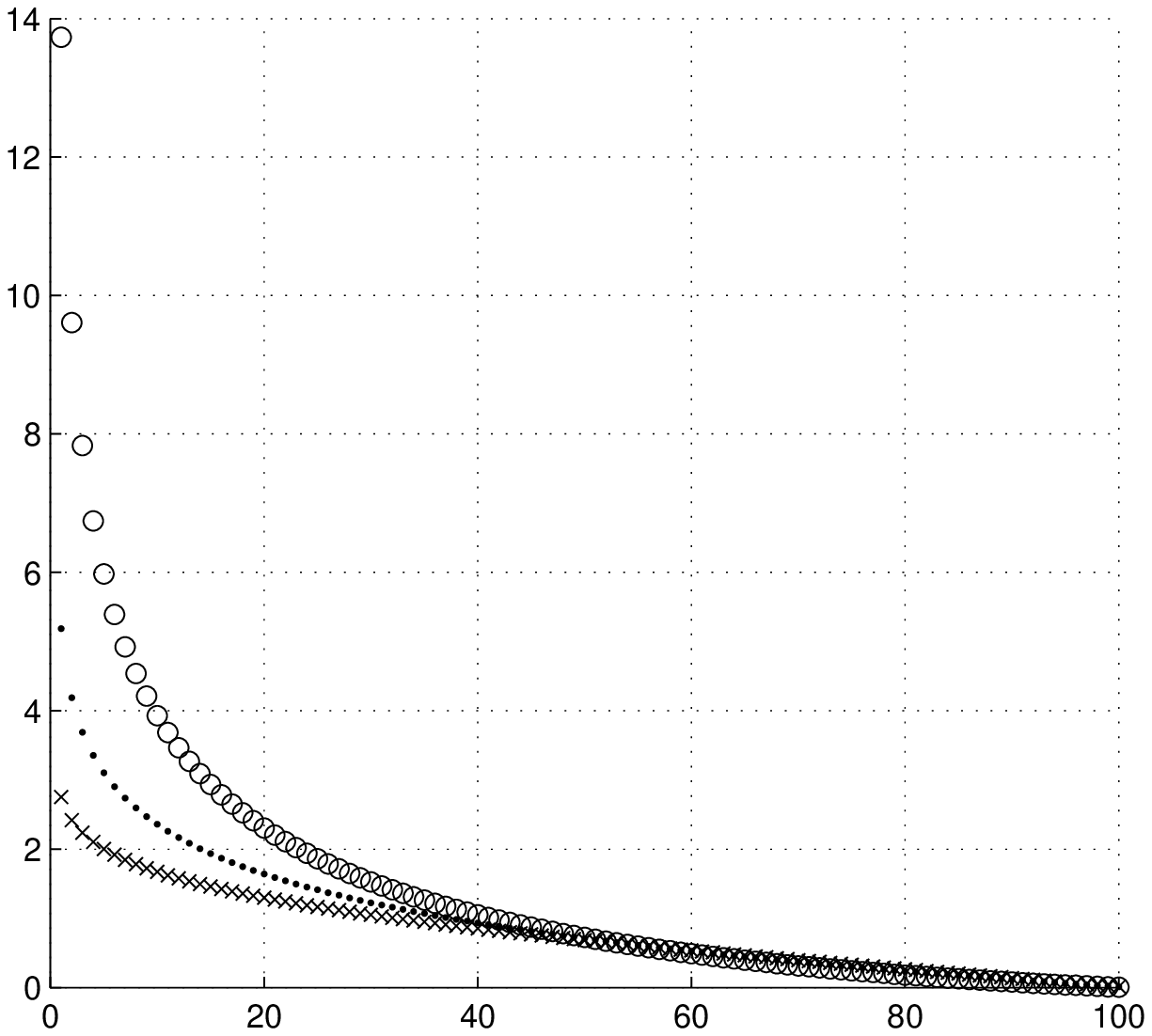}}
  \subfloat[$\log_{10}(\int_{\D}x_m^* d\theta_{p,p/n-1})$]{\includegraphics[width=.5\textwidth]{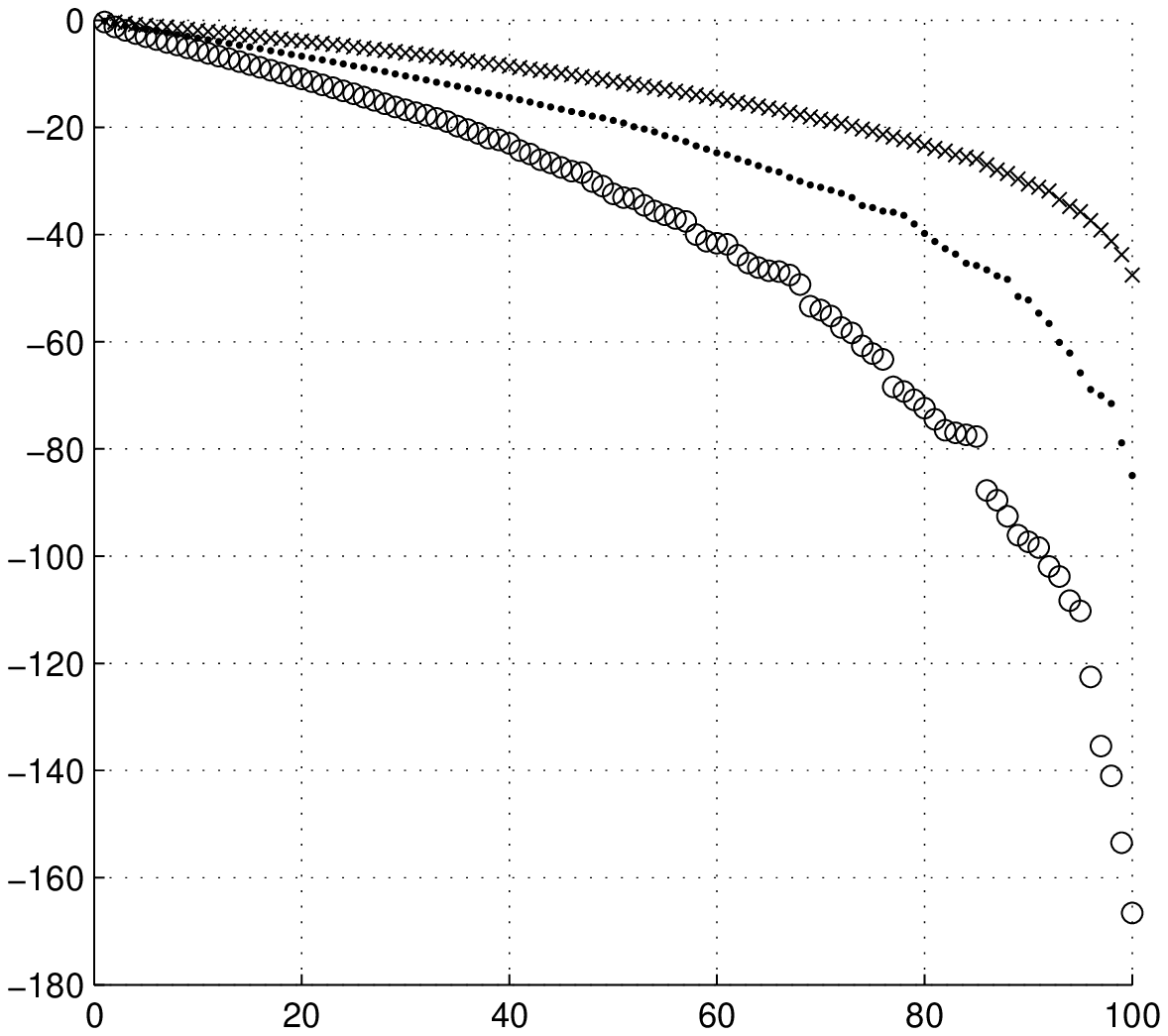}}\hfill
  \caption{Approximations of $n^{1/p}\cdot \int_{\D}x_m^* d\mu_p(x)$ (left) and $\log_{10}(\int_{\D}x_m^* d\theta_{p,p/n-1})$
  (right) for $n=100$, $p=1/2 (\circ)$, $p=1 (\bullet)$ and $p=2 (\times)$ based on sampling of $10^8$ random points.}
\label{fig:1}
\end{figure}

\subsubsection*{Acknowledgments}
I would like to thank to Stephan Dahlke, Massimo Fornasier, Aicke Hinrichs, Erich Novak  and Henryk Wo\'zniakowski for their interest
in this topic and the anonymous referees for their valuable comments and remarks, which helped to greatly improve the quality
of the presented paper. In particular, the proof of the second part of Theorem \ref{thm6} was suggested by one of the referees.
I acknowledge the financial support provided by the START-award ``Sparse Approximation and 
Optimization in High Dimensions'' of the Fonds zur F\"orderung der wissenschaftlichen Forschung (FWF, Austrian Science Foundation).


\end{document}